\documentclass[11pt, reqno]{amsart}
\usepackage{amsthm,amsmath,amsfonts,amssymb,color}

\usepackage[bookmarks]{hyperref}
\usepackage{comment}

\reversemarginpar

\usepackage{mathtools}
\mathtoolsset{showonlyrefs}

\newtheorem{thm}{Theorem}[section]
\newtheorem{prop}[thm]{Proposition}

\newtheorem{lemma}[thm]{Lemma}
\newtheorem{defn}[thm]{Definition}

\newtheorem{preremark}[thm]{Remark}
\newenvironment{remark}{\begin{preremark}\rm}{\medskip \end{preremark}}

\numberwithin{equation}{section}

\newcommand{\abs}[1]{\left\vert#1\right\vert}

\newcommand{\R}{\mathbb R}

\newcommand{\lt}{\tilde\theta}
\newcommand{\lp}{\tilde\phi}

\newcommand{\al}{\alpha}

\newcommand{\bt}{\beta}

\newcommand{\grp}{g^{\rho\phi}}

\DeclareMathOperator{\Vol}{Vol}

\newcommand{\eps}{\varepsilon}

\newcommand{\dd} {\mathrm{d}}

\DeclareMathOperator{\dv}{div}
\DeclareMathOperator{\Def}{Def}
\DeclareMathOperator{\Ric}{Ric}

\DeclareMathOperator{\curl}{curl}
\DeclareMathOperator{\two}{\mathrm I\mathrm I}

\DeclareMathOperator*{\diag}{diag}

\def\be{\begin{equation}}
\def\ee{\end{equation}}
\def\bs{\begin{split}}
\def\ess{\end{split}}

\newcommand{\tg}{\tilde g}
\newcommand{\tn}{\tilde \nabla}
\def\ag{\abs{\nabla\rho}}

\begin{document}
\title[Thin shell limit]{Thin shell limit and the derivation of the viscosity operator on the ellipsoid}

\author[Chan]{Chi Hin Chan}
\address{Department of Applied Mathematics, National Yang Ming Chiao Tung University,1001 Ta Hsueh Road, Hsinchu, Taiwan 30010, ROC}
\email{cchan@math.nctu.edu.tw}

\author[Czubak]{Magdalena Czubak}
\address{Department of Mathematics\\
University of Colorado Boulder\\ Campus Box 395, Boulder, CO, 80309, USA}
\email{czubak@math.colorado.edu}
  
  \author[Fuster]{Padi Fuster Aguilera}
\address{Department of Mathematics\\
University of Colorado Boulder\\ Campus Box 395, Boulder, CO, 80309, USA}
\email{padi.fuster@colorado.edu}

\begin{abstract}
 In this paper we derive four new candidates for an intrinsic viscosity operator on an ellipsoid by using the heuristic of the thin shell limit along the scaling direction of the ellipsoid.  We show that the general method of the thin shell limit through the asymptotic expansion depends on the averaging method used.  We consider both the homogeneous Navier and Hodge boundary conditions. We also obtain a geometric representation of these two boundary conditions.
\end{abstract}
%\date{\today}
\maketitle

\tableofcontents
\section{Introduction}
 
The thin shell limit is a method of deriving equations in lower dimensions from an equation in higher dimensions.  The method is used in several fields, in particular, in elasticity and fluid mechanics.  In elasticity, it can be applied to derive the Kirchhoff-Love model \cite{CiarletII}.  The first rigorous results on the Navier-Stokes equations in thin shell domains go back to the works of Raugel and Sell \cite{RaugelSell89, RS3, RS1, RS2}.  These works have been followed by many authors, see for example \cite{TemamZianeR, Avrin96, MTZ, TemamZiane, Iftimie, Miura1, Miura2, Miura3, Miura_stationary} and references therein.  In particular, Temam and Ziane \cite{TemamZiane} considered the non-stationary Navier-Stokes equation in a thin shell around the sphere and by taking the thickness of the shell to zero derived a corresponding equation on the sphere.  Then in a series of papers, Miura \cite{Miura1, Miura2, Miura3, Miura_stationary} has been studying the thin shell limit for a general surface.  Before obtaining the above rigorous results, Miura presented a heuristic derivation in \cite{Miura}.  That article is a motivation for this current work. 

The main idea is to take a vector field that would be a solution to the 3D problem, and write its asymptotic expansion based on the distance to the surface.  The common way, also done by Miura, is to consider an asymptotic expansion using a regular distance function.  In this article, the surface we consider is an ellipsoid, and in that setting, there is another natural way to measure the ``distance" from the ellipsoid.  

The ellipsoid $E$ can be viewed as a level set of a defining function $\rho$, where
\[
\rho(x,y, z)=\sqrt{\frac{x^2+y^2+a^2z^2}{a^2}}, \quad a>0.
\]
Then, we can let $E=\rho^{-1}(1)$.  It is natural to consider rescaled ellipsoids $E_{1+\varepsilon}=\rho^{-1}(1+\varepsilon)$, $\varepsilon>0$, and for $\varepsilon_0>0$ consider the Navier-Stokes equation in the thin shell near $E$ defined by
\[
\mathcal N(\varepsilon_0)=\cup_{0<\varepsilon <\varepsilon_0} E_{1+\varepsilon}.
\]
The boundary of $\mathcal N(\varepsilon_0)$ consists of $E$ and $E_{1+\varepsilon_0}$, and the idea is to take $\varepsilon_0$ to zero. Note, we could also consider rescaled ellipsoids in both directions of $E$, but we choose to follow \cite{TemamZiane} in the setup of the problem. 

Such domain and rescaled ellipsoids $E_{1+\varepsilon}$ are natural in the sense that the geometry of $E$ gets preserved modulo the scaling.  For example, the principal curvatures get scaled by $\frac 1{1+\varepsilon}$.

Since $\mathcal N(\varepsilon_0)$ is a bounded domain, we need to impose boundary conditions.  Temam and Ziane \cite{TemamZiane} use the Hodge condition, also referred to as the free boundary condition, and Miura \cite{Miura, Miura3, Miura_stationary} imposes the Navier boundary condition.  In this article, we consider both of these boundary conditions, and show that the asymptotic expansion in the scaling direction produces different operators than the ones obtained in \cite{TemamZiane} and in \cite{Miura, Miura3, Miura_stationary}.  

With the Hodge boundary condition, the Navier-Stokes equations obtained using the thin shell  limit in \cite{TemamZiane} are the equations where the Laplacian is the Hodge Laplacian.  Recall, the Hodge Laplacian is
\[
\dd\dd^\ast+\dd^\ast\dd,
\]
where $\dd$ is the exterior derivative, and $\dd^\ast$ its formal adjoint.

The Navier boundary condition in  \cite{Miura, Miura3, Miura_stationary} produces the operator
\be\label{divdef}
-2\dv \Def
\ee
where $\Def$ is the deformation tensor, which when written in coordinates is
\be\label{Defd}
 {(\Def u)}_{ij}=\frac 12 (\nabla_i u_j +\nabla_j u_i),
\ee 
where $\nabla$ is the Levi-Civita connection.  In general, the two operators are not the same, but related by a formula obtained by a computation using Ricci identities.  The formula is
 \be\label{BW1}
 -2\dv\Def u=\dd\dd^\ast u+\dd^\ast\dd u+\dd \dd^\ast u-2\Ric u,
 \ee 
 where $u$ is a $1$-form, and 
 \[
 \Ric u=\Ric (u^\sharp, \cdot)^\flat,
 \]
 with $\Ric$ the Ricci tensor, and $u^\sharp$ corresponding vector field obtained by using the musical isomorphism $\sharp$.  We refer to the operator \eqref{divdef} as the deformation Laplacian.  
 
 In the Euclidean space, since the curvature is identically zero, 
 the deformation Laplacian and the Hodge Laplacian are the same operator.  However, since they are not the same for a general Riemannian manifold, this is one of the main reasons there is no universal agreement what the ``correct" form of the equations should be in the setting of Riemannian manifolds.  
 
 There is also another type of Laplacian that we could consider.  It is the Bochner Laplacian, $-\dv \nabla$.  All three Laplacians are related by the following formula

 \be\label{BW}
 -2\dv\Def u=-\dv \nabla u -\Ric u+\dd \dd^\ast u=\dd\dd^\ast u+\dd^\ast\dd u+\dd \dd^\ast u-2\Ric u.
 \ee

 In this article, we arrive at yet four other types of operators.  In general, we take the point of view that we should look for an appropriate viscosity operator rather than only for the right choice of the Laplacian.  We refer the interested readers to \cite{Czubak2024} for more on this topic, and to the references therein.

 The derivation of the new operators is one of the main contributions of this article.  The operators are derived by considering the asymptotic expansion in the scaling direction, and by first considering the Navier boundary condition, and then the Hodge condition, and considering vector fields with or without divergence free conditions on $\R^3$ and on $E$.

 % The direction of the expansion is directly related to the direction the average is taken in the rigorous results.  

 The direction of the asymptotic expansion is connected to the direction the average is taken when the limiting equations are derived rigorously.  In the rigorous results, an average operator $Mv$ is typically considered, where
\[
Mv=\frac 1\eps \int_0^{\eps} v(y+\sigma N_y)d\sigma, \quad y\in S,
\]
where $S$ is the surface around which we consider the thin shell, and $N_y$ is the choice of the normal at $y$. If $v$ is assumed to be written as a series expansion in terms of the distance $\sigma$ from the boundary, i.e., 
\[
v=\sigma^\alpha U_\alpha,
\]
where $\alpha$ runs from zero to infinity, and $U_\alpha$ are assumed to be vector fields on $S$ translated in the normal direction, then we can see
\[
Mv=U_0+O(\eps).
\]
This and the fact that by considering the expansions in the scaling direction and obtaining different operators than the ones by considering the average in the normal direction suggests that the operators obtained by the thin shell limit depend not only on the boundary conditions, but also on the direction of the average taken.  

The average in the normal directions for domains with nonzero curvature with the Hodge boundary condition was only considered for the sphere \cite{TemamZiane} (see \cite{TemamZianeR} for rectangular domains).  Therefore, we also consider the Hodge boundary condition for the thin domain around the ellipsoid $E$, and perform the asymptotic expansion in the normal direction.  We show it produces the Hodge Laplacian, just like it did for the sphere.

In addition, we find what appears to be a novel way to interpret the Navier and Hodge conditions.  We show they both can be given a geometric interpretation.  In \cite{MitreaMonniaux} the relationship was derived between the Navier condition and the Hodge condition.  In this paper, we show that the Navier condition for the vector field $v$ can be interpreted as the projected Lie bracket of the vector field $v$ with the normal $N$, and the Hodge condition as the Lie derivative of the associated $1$-form $v^\flat$ in the direction of the normal $N$.  To put it simply, the two conditions can be both interpreted as Lie derivatives.  

Since we do not rigorously show convergence of solutions in any function space and only consider an asymptotic expansion, we claim that the derivation of the operators is a heuristic.  At the same time, we aim to justify what we can in this context.  For example, in the Appendix we provide a careful construction of an asymptotic expansion of a real analytic vector field in a thin shell.

Our computations are based on a formula obtained in \cite{CC25}.  Article \cite{CC25} gives a collection of Gauss formulas for different types of Laplacians on submanifolds.
In general, Gauss formulas provide a breakdown of an extrinsic object into the tangential and the normal part with a goal to relate the intrinsic quantities to the extrinsic ones. In the case of the Gauss formula for the Laplacian, we are interested in relating the extrinsic Laplacian to an intrinsic operator.   The formula that we use from \cite{CC25} is the formula for a surface embedded in $\R^3$.  In comparison, Miura \cite{Miura} works in Cartesian coordinates and makes use of the so-called G\"unter derivatives: Euclidean partial derivative operators orthogonally projected onto the tangent space of the surface.  Such derivatives also appear, for example, in finite element methods \cite{DziukElliott}.  We show our approach can also recover Miura's result in the case of the ellipsoid and the asymptotic expansion in the normal direction with the Navier boundary condition.

We are now ready to precisely state our main results.

\subsection{Main Results}
We establish the following theorem.
\begin{thm}\label{thm1}
Let $a,\eps>0$, and $E=\{{x^2+y^2+a^2z^2}=a^2\}$ be the ellipsoid embedded in $\R^3$, and $E_{1+\eps}=\{{x^2+y^2+a^2z^2}=a^2(1+\eps)\}$ the rescaled ellipsoid.  Given a vector field $v$ in the thin shell with boundary $E \cup E_{1+\varepsilon}$, where $v$ is also assumed to be tangential to all the rescaled ellipsoids $E_{1+b}, 0\leq b\leq \eps$, the heuristic of the thin shell limit using the asymptotic expansion of $v$ in the scaling direction produces the following operator on $E$ for the zeroth order term $U_0$ in the expansion

\be\label{o1}
-2\dv \Def U_0 +[c^3_{13}E_1, U_0]+2 \big ( c_{13}^3 \big )^2 U_0^1 E_1,
\ee
for the perfect Navier-slip boundary condition, and 
\be\label{o2}
(\dd^\ast \dd+\dd \dd^\ast) U_0^\flat+\mathcal L_{c_{13}^3 {E_1}} U_0^\flat,
\ee
for the homogeneous Hodge boundary condition, where 
\begin{itemize}
\item $\Def$ is the deformation tensor \eqref{Defd},
\item $[\cdot, \cdot]$ is the Lie bracket,
\item  $E_1$ is a unit vector field in the direction along the meridians on $E$,
\item $c^3_{13}=- \frac{1}{4} E_1 \big ( \log K_E \big )$, with $K_E$ the Gaussian curvature of $E$,
\item $\dd^\ast \dd+\dd \dd^\ast$ is the Hodge Laplacian on $E$,
\item $U^\flat_0$ is the $1$-form obtained by the musical isomorphism,
\item $\mathcal L$ is the Lie derivative of a $1$-form.
\end{itemize}
\end{thm}
\begin{remark}
We observe that the Navier boundary condition produces the deformation Laplacian just like the thin limit in the normal direction with two additional terms.  The additional terms both involve the coefficient function $c^3_{13}$.  The notation used here is due to how it originally arrives from the Lie bracket of $E_1$ with the normal (see Section \ref{prelim}). Interestingly enough, $c^3_{13}$ can be interpreted in several ways.  Since we seek an intrinsic formula, we use the intrinsic interpretation as above, in terms of the Gaussian curvature.

Similarly, the Hodge boundary condition produces the Hodge Laplacian, just like it does for the normal direction, with one additional term involving the Lie derivative.  In the Lie derivative, the vector field $c^3_{13}E_1$ again plays a role.  What is interesting to note is the connection of formula \eqref{o2} with the formula derived in \cite{CC25} for the surface of revolution and in \cite{CCY_22} specifically for the ellipsoid $E$.  Both formulas in these papers contained the Lie derivatives related to the Lie derivative appearing in \eqref{o2}.  The process of obtaining these formulas is different from the one here and yet the Lie derivative with respect to the vector field $c^3_{13}E_1$ makes an appearance.
\end{remark}
The next theorem is in the setting of a vector field $v$ in the thin domain, where $v$ is assumed to be tangential to the boundary, and is divergence free both on $\R^3$, and on the ellipsoid $E$.

\begin{thm}\label{thmd}
Let $a,\eps>0$, and $E=\{{x^2+y^2+a^2z^2}=a^2\}$ be the ellipsoid embedded in $\R^3$, and $E_{1+\eps}=\{{x^2+y^2+a^2z^2}=a^2(1+\eps)\}$ be the rescaled ellipsoid.  Given a divergence free vector field $v$ in the thin shell with boundary $E \cup E_{1+\varepsilon}$, such that $v$ is also divergence free on $E$, the heuristic of the thin shell limit using the asymptotic expansion of $v$ in the scaling direction with the Navier boundary condition produces the following operator on $E$ for the zeroth order term $U_0$ in the expansion
\be\label{o3}
-2\dv \Def U_0 +[c^3_{13}E_1, U_0],
\ee
and with the Hodge boundary condition, we obtain
\be\label{o4}
(\dd^\ast \dd+\dd \dd^\ast) U_0^\flat+\mathcal L_{c_{13}^3 {E_1}} U_0^\flat- 2 \big ( c_{13}^3 \big )^2 U_0^1 E^1.
\ee
 \end{thm}

\begin{remark}
The assumptions on the vector field $v$ in Theorem \ref{thm1} are motivated by both the rigorous works and the form of the asymptotic expansion that was taken in \cite{Miura}.  In the rigorous results, the average operator is obtained by working only with the tangential piece of the 3D vector field.  Similarly, in \cite{Miura} the expansion is obtained by expanding using the tangential vector fields, with the exception of them being tangential to the intermediate surface obtained by the normal direction, to the ``distorted" ellipsoids.

On the other hand, the assumptions in Theorem \ref{thmd} are motivated by the setting of the Gauss formulas, where the main object of interest is the vector field on the surface $E$, which is then extended to a vector field in the neighborhood of $E$.  
\end{remark}
\begin{remark}  We remark that since the sphere can be viewed as an ellipsoid with $a=1$, we can see what happens with the above formulas for the case of the sphere. In the case of the sphere $c^3_{13}=0$.  Hence both of the formulas \eqref{o1} and \eqref{o3} for the Navier boundary condition reduce to the deformation Laplacian, and for the Hodge boundary condition, formulas \eqref{o2} and \eqref{o4} reduce to the Hodge Laplacian. 
 
\end{remark}
 
Whether or not the above formulas can be rigorously derived is to be determined, and will be pursued in future works.   

We next have the theorem for the Hodge boundary condition for the thin domain
around the ellipsoid $E$ obtained by considering the normal direction.

\begin{thm}\label{thm2}
Let $a,\eps>0$, and $E=\{{x^2+y^2+a^2z^2}=a^2\}$ be the ellipsoid embedded in $\R^3$. For each $p \in E$, let $N_{p}$ be the outward pointing normal unit vector to the ellipsoid $E$ at $p$, and 
\begin{equation} 
\tilde E_{\sigma} = \big \{ p + \sigma N_p : p \in E     \big \}.
\end{equation}
Set

\begin{equation}\label{tubular}
\mathcal{T}(\eps) =  \bigcup_{\sigma \in (-\eps , \eps )} \tilde E_{\sigma} .
\end{equation}
Given a vector field $v$ in the thin shell $\mathcal{T}(\eps)$  with boundary $\tilde E_{-\eps} \cup \tilde E_{\varepsilon}$, where $v$ is also assumed to be tangential to every $\tilde E_{\sigma}, -\eps\leq \sigma\leq \eps$, and satisfies the homogeneous Hodge condition on the boundary, the heuristic of the thin shell limit using the asymptotic expansion of $v$ in the normal direction produces the Hodge Laplacian on $E$ for the zeroth order term $U_0$ in the expansion.
\end{thm}
\begin{remark}
This is the heuristic version for the ellipsoid of what was obtained rigorously for the sphere \cite{TemamZiane}.  The rigorous and the heuristic approach both produce the Hodge Laplacian.  We also note that the heuristic for the normal direction does not produce different formulas depending on the assumptions on $v$ regarding being always tangential and/or divergence free.  The theorem is stated for the simplicity for the simpler case of $v$ being tangential to all the $\tilde E_\sigma$, which again also resembles the heuristic set up in \cite{Miura}.
\end{remark}
\begin{remark}
  We also consider the Navier boundary condition and obtain the deformation Laplacian.  This is done in Section \ref{comp}, and is consistent with what was obtained in \cite{Miura}.
\end{remark}

The last result is the geometric formulation of the Navier and Hodge boundary conditions.  
\begin{thm}\label{thmb}
Let $M^n$ be a submanifold embedded in $\tilde M^{n+1}$, and $v$ a vector field on $\tilde M^{n+1}$ that is also tangential to $M$.  If $N$ is the choice of normal to $M$, then the perfect Navier slip boundary condition on $M$ is given by the tangential part of the Lie derivative of the velocity vector field $v$ with respect to $N$ set to zero.

In addition, if $n=2$, and $\tilde M^3=\R^3$, then the Hodge condition $\curl v\times N=0 $ for $v$ tangential to $M$ is equivalent to the (pullback to $M$) of the Lie derivative of the associated $1$-form $v^\flat$ with respect to $N$ set to zero.
\end{thm}
\begin{remark}
    The above statements will be made more precise when we introduce more notation.
\end{remark}
\begin{remark}
    Theorem \ref{thmb} could also cover the case of the Navier slip boundary condition with a friction coefficient.  See Section \ref{sect3} for details.
\end{remark}
The article is organized as follows. Section \ref{prelim} introduces the necessary geometric tools, including coordinates on the ellipsoid, the shape operator, and helpful formulas, as well as the setup for power series expansions. Section \ref{sect3} provides a geometric formulation of the Navier and Hodge boundary conditions, interpreting them via Lie derivatives. Section \ref{sect4} contains the proofs of the main results for the thin shell limit using asymptotic expansions in the scaling direction, addressing both Navier and Hodge conditions. Section \ref{sectHodge} focuses on the expansion in the normal direction under the Hodge condition. Section \ref{comp} offers a comparison with earlier works and discusses the Navier condition in the normal direction. Finally, Appendices A and B justify the power series expansions and review some relationships between forms and vector fields.

\section{Preliminaries}\label{prelim}
We begin with reviewing the metric and coordinates on the ellipsoid as they were defined in \cite{CCY_22}.
\subsection{Coordinates on the ellipsoid and on $\R^3$: scaling direction}\label{setupc}
We work with an ellipsoid $E$ defined by 
\[
E=\{ x^2+y^2+a^2z^2=a^2\},
\]
where $a>0$. $E$ can be parametrized by the map
\[
\Phi:  (0,\infty)\times (0,\pi)\times (-\pi,\pi) \to \R^3
\]
given by
\be\label{parametrize}
\Phi(\rho, \phi, \theta)=(a\rho\sin \phi \cos\theta, a\rho \sin\phi \sin\theta, \rho\cos \phi),
\ee
when $\rho=1$.  The map $\Phi$ induces coordinates on $\R^3\setminus \{0\}.$  The coordinate vector fields are
\begin{align}
\partial_\rho&=(a\sin \phi \cos\theta, a\sin\phi \sin\theta, \cos \phi),\nonumber\\
\partial_\phi&=(a\rho\cos\phi \cos\theta, a\rho\cos\phi\sin\theta, -\rho\sin\phi),\label{origphi}\\
\partial_\theta&=(-a\rho\sin\phi \sin\theta, a\rho\sin\phi\cos\theta, 0)\label{origtheta},
\end{align}
and the Euclidean metric is 
\be
\left(\begin{array}{ccc}
   g_{\rho \rho}&g_{\rho\phi}&0\\
   g_{\phi\rho }&g_{\phi\phi}&0\\
 0&0&g_{\theta\theta}\\
   \end{array}\right),
\ee
where
\begin{align}
g_{\rho\rho}&=a^2 \sin^2\phi+\cos^2\phi, \quad g_{\rho \phi}=g_{\phi\rho}=(a^2-1)\rho \sin\phi\cos\phi,\\
g_{\phi \phi}&=\rho^2\lambda^2, \qquad\qquad\qquad g_{\theta\theta}=a^2\rho^2\sin^2\phi,\\
\lambda^2&=a^2\cos^2\phi+\sin^2\phi.
\end{align}

 The inverse metric is 
\be
\left(\begin{array}{ccc}
   g^{\rho \rho}&g^{\rho\phi}&0\\
   g^{\phi\rho }&g^{\phi\phi}&0\\
 0&0&g^{\theta\theta}\\
   \end{array}\right),
\ee
\begin{align}
g^{\rho\rho}&=\frac {\lambda^2}{a^2},
 \quad\qquad g^{\rho \phi}=g^{\phi\rho}=\frac{1-a^2}{a^2\rho}\sin\phi\cos\phi=-\frac{g_{\rho\phi}}{a^2\rho^2},\label{usenow}\\
g^{\phi \phi}&=\frac1{a^2\rho^2}(a^2 \sin^2\phi+\cos^2\phi), \quad g^{\theta\theta}=\frac1{a^2\rho^2\sin^2\phi}\nonumber.
\end{align}

The metric on the ellipsoid $E$ is given by
\be
\left(\begin{array}{cc}
   g_{\phi \phi}&0\\
   0&g_{\theta\theta}\\
   \end{array}\right),
\ee
where $g_{\phi\phi}, g_{\theta\theta}$ are as above, with $\rho=1$.

Then, if $u$ is a vector field on $E$, 
\[
u=u^\phi \partial_\phi+u^\theta \partial_\theta=u^1E_1+u^2E_2,
\]
with
\[
E_1=\frac {\partial_\phi}{\lambda},\quad E_2=\frac{\partial_\theta}{a\sin\phi},
\]
and
\[
u^1=\lambda u^\phi,\quad u^2=a\sin\phi u^\theta.
\]
Next, if $v$ is a vector field on $\R^3$, then in general
\[
v=v^\rho\partial_\rho+v^\phi \partial_\phi+v^\theta \partial_\theta=v^1E_1+v^2E_2+v^3E_3,
\]
with
\be\label{frame_comp}
v^1=\frac{v_\phi}{\rho\lambda}=\frac{g_{\phi\rho}v^\rho+g_{\phi\phi}v^\phi}{\rho\lambda},\quad v^2=\frac{v_\theta}{a\rho\sin\phi}=a\rho\sin\phi v^\theta,\quad v^3=\frac{a v^\rho}{\lambda},
\ee
and
\[
E_1=\frac {\partial_\phi}{\rho\lambda},\quad E_2=\frac{\partial_\theta}{a\rho\sin\phi},\quad E_3=N=\frac{\nabla \rho}{\ag}=\frac \lambda a\partial_\rho+\frac a\lambda g^{\rho\phi}\partial_\phi.
\]

In this frame, we also have
\begin{align}
\Gamma_{33}^1=-\Gamma_{31}^3=c_{13}^3=\frac{a^2g^{\rho\phi}}{\lambda^3},\label{c313}
\end{align}
where $c^3_{13}$ is the component function in
\[
[E_1, N]=c_{13}^\alpha E_\alpha.
\]
In general,
\[
[E_\alpha, E_\beta]=c^\gamma_{\alpha\beta}E_\gamma,
\]
and
\[
\Gamma_{\alpha\beta}^\gamma=\frac 12(c^\gamma_{\alpha\beta}-c^\beta_{\alpha\gamma}-c^\alpha_{\beta\gamma}),
\]
which gives
\be\label{flipit}
\Gamma_{\alpha\beta}^\gamma=-\Gamma_{\alpha\gamma}^\beta.
\ee
Another computation shows
\be
\tn_N E_1=c_{31}^3N,\quad \tn_N E_2=0,\quad\tn_NN=c^3_{13}E_1.\label{tnN}
% \tn_{E_1} N&=c_{13}^1E_1,\quad
%\tn_{E_2} N=c_{23}^2E_2,\label{Weq1}
\ee
And we can show that
\be\label{gammas0}
\Gamma_{11}^1= \Gamma_{11}^2= \Gamma_{12}^1=\Gamma_{12}^2=\Gamma^1_{21}=0, \quad \Gamma^2_{21}=c^2_{21}=\frac {\cot\phi} \lambda, \quad\mbox{on E}.
\ee
This results in
\be\label{helpful5}
\nabla_{E_1}X=E_1(X^j)E_j,\quad \nabla_{E_2}E_1=\frac {\cot\phi} \lambda E_2,
\ee
where we use the Einstein notation to sum with respect to $j$.  In general, Roman indices sum from $1$ to $n$, where $n$ is the dimension of the submanifold, and Greek indices sum from $1$ to $n+1$, the dimension of the ambient Euclidean space, with $E_{n+1}$ always denoting the normal vector field $N$.  The only exception is when we use infinite series, and in that case we sum to infinity (see Section \ref{powerseries}).  If there is a possibility of confusion, we use an explicit sum.
\subsection{Shape operator}
We review briefly some properties of the shape operator  \cite[Ch.8]{Lee_RG}.  

We use this  background in the context of the boundary of the thin-shell domain being viewed as an embedded submanifold in $\R^3$.  So for example, for 
\[
\Omega=\{ (\rho,\phi,\theta): 1<\rho<1+\eps\},
\]
\[
\partial \Omega=E\cup E_{1+\eps},\quad E_{1+\eps}=\{\rho=1+\eps\},
\]
we can let $M=E$ below or $M=E_{1+\eps}$.

Let $M^2\hookrightarrow \R^3$, and $\mathfrak X(M)$ denote the smooth vector fields on $M$.  Then if $X, Y \in \mathfrak X(M)$, $X, Y$ can be extended to be vector fields in a neighborhood of $M$ in $\R^3$. Then the Gauss
formula on $M$ reads
\be\label{gauss}
\tn_X Y=\nabla_X Y+\two(X,Y).
\ee
Here $\tn$ is the Euclidean connection on $\R^3$, $\nabla$ is the Levi-Civita connection on $M$, and $\two : \mathfrak X(M) \times \mathfrak X(M) \to \Gamma (NM)$ is the second fundamental form, which sends two vectors fields on $M$ to smooth sections of the normal bundle of $M$.  
  
Let $N$ be a choice of a unit normal, $N\in \Gamma(NM)$.  Then the scalar second fundamental form is given by
\be
h(X,Y)=\tilde g(N, \two(X,Y)).
\ee
The shape operator
\[
s:  \mathfrak X(M) \to  \mathfrak X(M),
\]
is obtained by raising the index of $h$, 
\be\label{shape_def}
\tilde g(s X, v)=h(X,v).
\ee
The eigenvalues of $s$, $\kappa_i$ are called the principal curvatures, and at each $p\in M$, we have an orthonormal basis $\{E_i\}$ of eigenvectors.  For example, for $M=E\hookrightarrow \R^3$,  the principal curvatures are
\be\label{kappas}
\kappa_1=-\frac a{\lambda^3}, \quad \kappa_2=-\frac 1{a\lambda},
\ee
and the corresponding eigenvectors are $E_1$ and $E_2$ as introduced above.

Finally, we write down the Weingarten equation for a hypersurface
\be\label{Weq}
sX=-\tilde \nabla_X N.
\ee

\subsection{Divergence free conditions}
We assume the vector field is tangential to the ellipsoid $E$.  Then, if $v=v^\alpha E_\alpha$, 
\be
v^3=0\ \mbox{at}\ p\in E.
\ee
We use this repeatedly in the computations below.

If the vector field $v$ is both divergence free on $\R^3$ and on $E$, then we can show
\be\label{divfree}
N(v^3)=v^1c^3_{13} \quad\mbox{at}\ \ p\in E.
\ee

\subsection{Coordinates on the ellipsoid and on $\R^3$: normal direction}\label{setupd}
For the purpose of showing Theorem \ref{thm2} and comparing the results of this paper to what was obtained in \cite{Miura}, it is helpful to use a coordinate system adapted to the expansion in the normal direction.  We set this up now.

We consider a tubular neighborhood $\mathcal T(\eps)$ around the ellipsoid $E=\{ x^2+y^2+a^2z^2=a^2\}$ obtained by going distance $\eps$ in the normal direction, $\pm N_p$, $p\in E$.

At each $p\in E$, $p=(x_0, y_0, z_0)$,
\be\label{Np}
N_p=\frac{(x_0, y_0, a^2z_0)}{\sqrt{x_0^2+y_0^2+a^4z_0^2}},
\ee
and we observe that if $p\in E$, using the parametrization \eqref{parametrize}, $$p=(a\sin\phi \cos \theta, a\sin\phi \sin\theta, \cos \phi),$$ then
\be\label{Np1}
x_0^2+y_0^2+a^4z_0^2=a^2\lambda^2.
\ee
It might be helpful here to rename the variables $(\phi, \theta)$ as $(\lp, \lt)$.  We then have that for $p\in E$, $(\phi, \theta)=(\lp, \lt)$, and \eqref{Np1} still holds with $\lambda=\lambda(\lp, \lt)$.

We parametrize $\mathcal T(\eps)$ by a map $X(\sigma, \lp, \lt)$ so that 
\[
X(0,\lp, \lt)=(a\sin\lp \cos \lt, a\sin\lp \sin\lt, \cos \lp)\in E,
\]
and for general $x\in \mathcal T(\eps)$, 
\[
x=p+\sigma N_p,\quad p\in E, \ \sigma \in(-\eps,\eps),
\]
so using \eqref{Np} and \eqref{Np1}, we can compute
\begin{align}
N_p&=\frac 1{\lambda}(\sin\lp \cos \lt, \sin\lp \sin\lt, a\cos \lp),\label{newNp}\\
x&=X(\sigma, \tilde\phi,\tilde\theta):=\left((1+\frac{\sigma}{a\lambda})a\sin\lp \cos\lt, (1+\frac{\sigma}{a\lambda})a\sin\lp \sin\lt, (1+\frac{a\sigma}{\lambda})\cos\lp \right).\nonumber
\end{align}
We then can compute
\begin{align*}
\partial_\sigma|_x&=\frac 1{\lambda}\left(\sin\lp \cos\lt,  \sin\lp \sin\lt,  a\cos\lp \right),\\
\partial_{\lp}|_x&=(1+\frac{a\sigma}{\lambda^3})\left(a\cos\lp \cos\lt,  a\cos\lp \sin\lt, - \sin\lp \right), \\
\partial_{\lt}|_x&= (1+\frac{\sigma}{a\lambda})\left(-a\sin\lp \sin\lt,  a\sin\lp \cos\lt, 0 \right).
\end{align*}
Comparing to \eqref{newNp} and \eqref{origphi}-\eqref{origtheta}, we observe that we can think of the above vector fields as
\begin{align}
\partial_\sigma|_x&=N_{\pi(x)}, \quad \pi(x)=\pi(p+\sigma N_p)=p,\label{Npt}\\
\partial_{\lp}|_x&=\partial_\phi|_{(1,\phi,\theta)}\cdot (1+\frac{\sigma a}{\lambda^3}),\label{newphi}\\
\partial_{\lt}|_x&= \partial_\theta|_{(1,\phi,\theta)}\cdot (1+\frac{\sigma}{a\lambda})\label{newtheta}.
\end{align}
This is helpful when computing the metric, and we arrive that in these coordinates, the Euclidean metric in $\mathcal T(\eps)$ can be written as
\[
g=\diag (1, g_{\lp\lp}, g_{\lt\lt})=\diag (1, (1+\frac{\sigma a}{\lambda^3})^2\lambda^2, (1+\frac{\sigma}{a\lambda})^2a^2\sin^2\lp).
\]

\subsection{Helpful formulas}\label{helpful}
For future reference, we record here some helpful formulas, which hold on the ellipsoid $E$. First, we can write
\[
(1-a^2)\sin^2\phi=\lambda^2-a^2,\quad (1-a^2)\cos^2\phi=1-\lambda^2.
\]
Then
\be\label{helpful1}
a^2(g^{\rho\phi})^2=\frac{\lambda^2-a^2-\lambda^4+a^2\lambda^2}{a^2}.
\ee
We also have
  \[
  \partial_\phi (a^2g^{\rho\phi})=1+a^2-2\lambda^2,
  \]
  so 
  \be\label{helpful2}
  E_1(a g^{\rho\phi})=\frac 1{\lambda}\partial_\phi(ag^{\rho\phi})=\frac 1{a\lambda}+\frac a\lambda-2\frac \lambda a,
  \ee
  and
  \be\label{helpful3}
  E_1(\lambda )=\frac{\partial_\phi}{\lambda}(\lambda)=\frac{a^2g^{\rho\phi}}{\lambda^2}.
  \ee
  Then by \eqref{helpful2} and \eqref{helpful3}
  \be\label{helpful4}
  E_1 (c^3_{13})=E_1(\frac{a^2\grp}{\lambda^3})=- \frac{2}{\lambda^4} (1 + a^2) + \frac{1}{\lambda^2} + \frac{3a^2}{\lambda^6}.
  \ee
  Finally,
  \be\label{c313squared}
  (c^3_{13})^2=\frac{a^4(\grp)^2}{\lambda^6}=\frac 1{\lambda^4}-\frac{a^2}{\lambda^6}-\frac 1{\lambda^2}+\frac{a^2}{\lambda^4}.
  \ee

\subsection{Power series expansion}\label{powerseries}
We are interested in expanding the vector field $v$ on $\R^3$ as a power series of the vector fields on the ellipsoid $E$.  We consider doing this in two ways: in terms of the distance $d$ from the surface and in terms of the scaling parameter $\rho$.  

The following is the expansion of $v$ in the powers of the scaling parameter $\rho$ with the vector fields on $E$ extended by parallel transport in the direction of the vector field $\partial_\rho$.  It is given by
\be\label{ppo1}
v=(\rho-1)^\alpha v^\rho_\alpha\partial_\rho+ (\rho-1)^\alpha \frac 1\rho U_\alpha, \quad U_\alpha=U^\phi_\alpha\partial_\phi+U^\theta_\alpha \partial_\theta,
\ee
where we sum over the repeated index $\alpha$, starting at $0$, and $U_\alpha$ depend only on $\phi$ and $\theta$. Such expansion can be rigorously justified for real analytic vector fields.  We refer the reader to the Appendix \ref{justify}.  In addition, if $v$ is assumed to be tangential to all the rescaled ellipsoids, then we can let $v^\rho_\alpha=0$ for all $\alpha$.

Similarly, for the expansion in the normal direction, we can write
\be\label{exp:sigma}
v=\sigma^\alpha U_\alpha,
\ee
where $\sigma$ is the (signed) distance from the ellipsoid, and each $U_\alpha$, initially a vector field on $E$, is extended by parallel transport in the normal direction. The justification for this expansion is also addressed in the Appendix
 \ref{justify}.

\section{Proof of Theorem \ref{thmb}: geometric formulation of boundary conditions}\label{sect3}

In this section we discuss the Navier and Hodge boundary conditions, and show that they can be formulated using the language of geometry.  
 
Throughout this we use the notation that for vector fields $X=X^\alpha E_\alpha$ and $v=v^\alpha E_\alpha$, the covariant derivative of $v$ in the direction of $X$ can be expressed as
\[
\tn_X v=X^\alpha\tn_{\alpha} v^\beta E_\beta,
\]
where 
\[
\tn_\alpha v^\beta=E_\alpha (v^\beta)+\Gamma_{\alpha\gamma}^\beta v^\gamma.
\]
Similarly,
\[
\tn_\alpha v_\beta=E_\alpha (v_\beta)-\Gamma_{\alpha\beta}^\gamma v_\gamma,
\]
and
\be\label{same}
\tn_\alpha v^\beta=\tn_\alpha v_\beta
\ee
by \eqref{flipit} and since in an ON frame $v^\beta=v_\beta$.

The (perfect) Navier slip boundary condition is the no penetration condition $u\cdot N=0$, which means the vector field is tangential at the boundary, and in addition, we have
\be\label{navier_c}
(TN)_{tan}=0,
\ee
where $T$ is a stress tensor
\[
T_{\al\bt}=2\Def_{\al\bt} (v)-p\delta_{\al\bt},
\]
and in \eqref{navier_c}, we think of $T$ acting on $N$ as a matrix, resulting in a vector, which we then project onto the tangent space to our manifold.

The Hodge condition is the no penetration condition together with
\[
\curl v\times N=0.
\]
Note, as written this condition is clear only for a 3D domain.

In \cite{MitreaMonniaux}, the authors showed the Navier condition can be related to the Hodge condition by establishing the following relation.
\be\label{mm1}
(TN)_{tan}=\curl v\times N+2sv,
\ee
where $s$ is the shape operator.  In addition, it was shown that 
\be\label{mm}
\curl v\times N=\iota_N \dd v^\flat,
\ee
where $\iota$ denotes the interior multiplication. This implies that the right hand side could be used as a generalization of the Hodge condition to other dimensions (see \cite{MitreaMonniaux} for details).  We note there are appropriate identifications made in \eqref{mm1} and \eqref{mm}.  For example, forms are canonically identified with vector fields. 

We now show that these conditions can be formulated using Lie derivatives.  This will also lead to another way we can derive the formula \eqref{mm1} and generalize the Hodge condition to other dimensions.
\subsection{Navier condition as a Lie derivative}
We can write $T$ as 
\[
T=2\Def(v)-pI,
\]
where $I$ is the identity matrix.  So by linearity, we have
\be\label{tndef}
(TN)_{tan}=2(\Def(v)N)_{\tan}.
\ee

\begin{prop}\label{NavierLie}  The perfect Navier slip boundary condition for a tangential vector field is equivalent to the tangential part of the Lie derivative of the velocity vector field or equivalently, the tangential part of the Lie bracket.  More precisely, let $M^n$ be a submanifold embedded in $\tilde M^{n+1}$.  If $N$ is the choice of normal to $M$, then 
\be\label{LieNavier}
 (TN)_{tan}=(\mathcal L_N v)_{tan}=[N,v]_{tan}.
\ee
\end{prop}
\begin{proof}
It holds for any vector fields $X, Y$ that 
\[
\mathcal L_X Y=[X,Y],
\]
so the second equality in \eqref{LieNavier} holds.  We now show the first one. 

By definition, for a vector field $v$,
\[
\Def(v)(X,Y)=\frac 12\{ g(\tn_Xv,Y)+g(\tn_Yv,X) \},
\]
where $\tn$ denotes the Levi-Civita connection in $\tilde M^{n+1}$. So in the ON frame $\{E_\alpha \}$
\be\label{def_cor}
\Def(v)_{\alpha\beta}=\frac{\tn_\alpha v_\beta+\tn_\beta v_\alpha}{2}.
\ee
 Now, if we think of $\Def v$ acting as a matrix, using \eqref{def_cor} and assuming it is written in an orthonormal frame, this produces
\be\label{matrix_action}
\Big(\frac{\tn_\alpha v_\beta+\tn_\beta v_\alpha}{2}\Big)N=\sum_{\alpha}\frac{\tn_\alpha v_{n+1}+\tn_{n+1} v_\alpha}{2}E_\alpha.
\ee
 
So if we only want the tangential component, then
\be\label{tang}
(TN)_{tan}= \sum_{i=1}^n\tn_i v_{n+1} E_i+\sum_{i=1}^n\tn_{n+1} v_iE_i,
\ee
where we used \eqref{tndef}.

We now show that this is equivalent to $[N, v]_{tan}$. Let $p\in M$ and assume we are evaluating below at $p$. We have
\be
[N, v]=\tn_N v-\tn_vN=\sum_\alpha\tn_{n+1}v^\alpha E_\alpha-\tn_v N.
\ee
Projecting gives
\be\label{liet}
[N, v]^T=\sum_{i=1}^n\tn_{n+1}v^i E_i+sv=\sum_{i=1}^n\tn_{n+1}v^i E_i+\sum_{i=1}^nv^i\kappa_i E_i,
\ee
where $s$ is the shape operator, and we now further assume that $\{E_i\}$ consists of the eigenvectors at $p$ of the shape operator.

Now
\[
\tg(v, N)=0
\]
for points on $M$.  Hence on $M$
\[
E_i g(v,N)=0,
\]
for any members of the ON frame of $M$, which also are (local) vector fields on $\tilde M^{n+1}$. Then by the properties of the connection $\tn$ on $\tilde M^{n+1}$,
\[
0=g(\tn_{E_i}v, N)+g(v, \tn_{E_i}N),
\]
or equivalently (note, $i$ is fixed, and we do not sum with respect to $i$)
\be\label{product}
g(\tn_{E_i}v, N)=-g(v, \tn_{E_i}N)=g(v, sE_i)=g(v, \kappa_i E_i)=v^i\kappa_i ,
\ee
but the left hand side is also equal to
\[
g(\tn_{E_i}v, N)=g(\tn_i v^\alpha E_\alpha, N)=\tn_i v^{n+1}.
\]
So by \eqref{product}
\[
\tn_i v^{n+1}=v^i\kappa_i.
\]
Plugging this into \eqref{liet} we obtain
\[
[N, v]^T=\sum_{i=1}^n\tn_{n+1}v^i E_i+\sum_{i=1}^n\tn_i v^{n+1} E_i,
\]
which is \eqref{tang} by \eqref{same}.
\end{proof}
The Navier slip boundary condition can also include the friction coefficient $\gamma\geq 0$, and then it would read as
\[
 (TN+\gamma v)_{tan}=0.
\]
 From Proposition \ref{NavierLie}, we can then write it as
 \[
 ([N,v]+\gamma v)_{tan}=0.
\]

\subsection{Hodge condition as a Lie derivative}
Here we would like to relate the Hodge condition
\[
\curl v\times N=0
\]
to the Lie derivative of the form $v^\flat$
\be\label{relate}
i^\ast_{M} (\mathcal L_N v^\flat)=0.
\ee
Note, in the case of $1$-forms, the Lie derivative will produce a $1$-form on the ambient manifold, so in order to obtain a $1$-form on the submanifold, we need to use the pullback by the inclusion map; the pullback here is the analog of the tangential projection of vector fields.

\begin{prop}\label{HodgeLie} Let $M$ be a surface embedded in $\R^3$, $i_M : M\hookrightarrow \R^3$ the inclusion map, and  $v$ a vector field on $\R^3$ such that $v\in \mathfrak X(M)$.  If $N$ is the choice of the unit normal to $M$, defined in a neighborhood in $\tilde M$, then 
\be\label{LieHodge}
i^\ast_{M} (\mathcal L_N v^{\flat_{\R^3}})=( \curl v\times N)^{\flat_M}.
\ee
\end{prop}
\begin{proof}
For simplicity, we drop the subscripts on $\flat$ operators in the proof.  Let $v$ be a vector field on $\R^3$ that is also tangential to $M$. 
Formula \eqref{LieHodge} is a consequence of the Cartan's formula
\[
\mathcal L_N v^\flat=\iota_N \dd v^\flat+\dd\iota_N v^\flat.
\]
First, since $ v\in \mathfrak X(M)$, $0=g(v,N)=\iota_N v^\flat$ on $M$. We claim this implies $i^\ast_M \dd\iota_N v^\flat=0$.  To see that, since $\dd\iota_N v^\flat$ is a $1$-form on $\R^3$, it can be written locally around each $p\in M$ as
\[
 \dd\iota_N v^\flat=\lambda_{\alpha} E^\alpha,
\]
for some smooth functions $\lambda_{\alpha}$.  Then
\[
(i^\ast \dd\iota_N v^\flat)(p)=\lambda_{i}(p) E^i|_p, \quad p\in M,
\]
and
\[
\lambda_i(p)=(\dd\iota_N v^\flat)(E_i|_p)=E_i|_pg_q(v,N)=0,\quad p, q\in M.
\]

% let $X\in \R^3$, and consider
%\[
%(\dd\iota_N u^\flat)(X)=X g(u,N)=
%\]
%
%
%and 
%  
Therefore \eqref{LieHodge} will follow from 
\be\label{eqnop}
\iota_N \dd v^\flat=(\curl v \times N)^\flat,
\ee
which is \eqref{mm} shown in \cite{MitreaMonniaux}.  We show this below for completeness using a different method of computation.
First using (see e.g., \cite{Jost})
\[
\dd v^\flat=E^\al\wedge \tn_{E_\alpha}v^\flat,
\]
we compute
\begin{align*}
\iota_N \dd v^\flat&=\iota_N (E^\al\wedge \tn_{E_\alpha}v^\flat)\\
&=E^\al(N)\tn_{E_\al}v^\flat-(\tn_{E_\al}v^\flat)(N)E^\al\\
&=\tn_Nv^\flat-\tn_{\al}v_3 E^\al\\
&=(\tn_3 v_i-\tn_{i}v_3) E^i,
\end{align*}
where we used that $\tn_N v^\flat=\tn_3 v_\alpha E^\alpha$ to go to the last line.
 
Next, using that for $1$-forms $\alpha, \beta$, $(\alpha ^\sharp \times \beta^\sharp)^\flat=\star(\alpha \wedge \beta)$, and that $\curl v=(\ast \dd v^\flat)^\sharp$ (see  \eqref{crossandwedge} and \eqref{curlandd}), we have
\[
(\curl v \times N)^\flat=\left((\star \dd v^\flat)^\sharp\times N\right)^\flat
=\star(\star \dd v^\flat \wedge N^\flat).
\]
To compute the last term, we use $N^\flat=E^3$ and
\[
\star (E^1\wedge E^2)=E^3,\quad\star (E^1\wedge E^3)=-E^2,\quad \star (E^2\wedge E^3)=E^1.
\]
This gives
\[
\star(\star \dd v^\flat \wedge N^\flat)=(\tn_3 v_1-\tn_1 v_3)E^1+(\tn_3 v_2-\tn_2 v_3)E^2,
\]
which is  exactly $\iota_N \dd v^\flat$ above.

\end{proof}
Theorem \ref{thmb} now follows from Propositions \ref{NavierLie} and \ref{HodgeLie}.  We finish with the following remark.
\begin{remark}
In \cite{CC25} the first two authors obtained a relationship between the Lie derivative of a form and the Lie derivative of a vector field.  In the specific case of the Lie derivative in the direction of the normal, it becomes (see \cite[Lemma 2.8]{CC25})
\be\label{NH}
((\mathcal L_{N}v^\flat)^\sharp)^T=(\mathcal L_N v)^T-2sv.
\ee
Then from \eqref{LieNavier} and \eqref{LieHodge}, this is another way we can obtain formula \eqref{mm1}.
 
\end{remark}
\section{Proof of Theorems \ref{thm1} and \ref{thmd}: expansion in the scaling direction}\label{sect4}
In this section we prove Theorems \eqref{thm1} and \eqref{thmd}.  We first start with a tangential part of the Gauss formula for a Laplacian of a vector field on $\R^3$.  The goal is to write everything in terms of the ON frame on the ellipsoid $E$ plus some terms that involve derivatives in the scaling direction that points away from $E$.
 \subsection{Viscosity operator: deriving an intrinsic formula}
In this section we work with a formula from \cite{CC25}.  We rely on the computations from Section \ref{prelim}.  

The starting point is the following theorem.

\begin{thm}\cite{CC25}\label{gausst} Let $n\geq 2$, and let $M$ denote an embedded hypersurface in $\R^{n+1}$, $M\hookrightarrow \R^{n+1}$, $\dim M=n$, and let $v$ be a vector field on $M$, $p \in M$, and $v$ be an extension of $v$ to a neighborhood of $p$ in $\R^{n+1}$, still denoted by $v$. If $T$ denotes the projection onto the tangent space $T_pM$,  then 

\be\label{project}
(-\dv \tilde \nabla v)^T=\nabla^\ast  \nabla v-\Ric v+nH[N,v]^T-(\tilde \nabla_N\tilde\nabla_N v)^T+\nabla_{\tn_N N}v,
\ee
where 
\begin{itemize}
\item $\tn$ denotes the Euclidean connection on $\R^{n+1}$; Hence $-\dv\tilde \nabla v$ is the standard Laplacian on $\R^{n+1}$;
\item $\nabla$ denotes the Levi-Civita connection on $M$,
\item $\Ric v=(\Ric (v,\cdot))^\sharp$, where $\Ric(\cdot, \cdot)$ is the Ricci tensor on $M$,
\item $H$ is the mean curvature of $M$,
\item $N$ is the choice of the unit normal to $M$, defined in a neighborhood of $p\in M$,
\item $[\cdot, \cdot]$ is the Lie bracket,
\end{itemize}
and it is assumed all the terms are evaluated at $p \in M$.
\end{thm}
We specialize the formula to $M=E$, and $\R^3$.  We start with the Bochner Laplacian the Ricci operator on $E$.

\subsubsection{Bochner Laplacian and the Ricci operator}
Here, when written for the j'th component in the ON frame $\{E_i\}$ introduced in Section \eqref{prelim}, we have
  \begin{align*}
  \nabla ^i \nabla_i v^j&=\nabla_i \nabla_i v^j,\\
  (\Ric v)^j=\kappa_1\kappa_2 v^j&=\frac 1{\lambda^4}v^j,
  \end{align*}
  by \eqref{kappas}.
  
 \subsubsection{Lie bracket}
 Here, since $n=2$
 \[
 2H=\kappa_1+\kappa_2,
 \]
 and on the ellipsoid $E$ 
 \be\label{liebracket_gen}
 \bs
 [N,v]^T=(\tn_Nv-\tn_vN)^T&=(N(v^\alpha)E_\alpha+v^\alpha\tn_NE_\alpha-\tn_vN)^T\\
 &=(N(v^\alpha)E_\alpha+v^i\tn_NE_i)^T+sv=\sum_i(\kappa_i v^i +N(v^i))E_i.
 \end{split}
  \ee
 \subsubsection{Double $N$ covariant derivative}
 We start by writing $N(h)$ in terms of $E_1(h)$, for some function $h$.  More precisely
 \be\label{Nintrinsic}
 N(h)=\frac{\lambda}a\partial_\rho h+\frac a\lambda g^{\rho\phi}\partial_\phi h=\frac{\lambda}a\partial_\rho h+a\rho g^{\rho\phi}E_1 h.
 \ee
Next, a computation using divergence free condition, both on $\R^3$ and on $E$, shows
 \begin{align}
 -(\tn_N\tn_N  v)^T&=-N(N(v^i))E_i-v^1(c^3_{13})^2E_1.\label{use?}
 \end{align}
 In general, without assuming the two divergence free conditions, we have
 \begin{align}
 -(\tn_N\tn_N  v)^T&=-N(N(v^i))E_i-2N(v^3)c^3_{13}E_1+v^1(c^3_{13})^2E_1\nonumber\\
 &=-N(N(v^i))E_i+v^1(c^3_{13})^2E_1\label{useg},
\end{align}
if the vector field is assumed to satisfy $v^3=0$ identically. 

Next,
 \begin{align*}
 N(N(v^i))&=\frac{\lambda}a\partial_\rho (Nv^i)+a\rho g^{\rho\phi}E_1 (Nv^i)\\
&=\frac{\lambda^2}{a^2}\partial_\rho^2 v^i + \lambda \rho g^{\rho\phi}\partial_\rho E_1 v^i\\
&\qquad+a\rho g^{\rho\phi}\left(E_1( \frac{\lambda}a)\partial_\rho v^i+\frac \lambda a E_1\partial_\rho v^i+E_1(a\rho g^{\rho\phi})E_1 v^i+a\rho g^{\rho\phi}E_1E_1 v^i\right).
  \end{align*}
  
  Above we need  
  \[
  \partial_\rho E_1 v^i=\partial_\rho \frac{\partial_\phi}{\rho \lambda}v^i=-\frac{1}{\rho^2\lambda}\partial_\phi v^i+\frac 1{\rho\lambda}\partial_\phi \partial_\rho v^i=-E_1 v^i+E_1 \partial_\rho v^i,
  \]
  where we use that we are evaluating the action of $E_1$ on $E$, so we can set $\rho=1$ in $E_1$.
  We plug in and evaluate on $E$ to obtain
 \[
 N(N(v^i))=\frac{\lambda^2}{a^2}\partial_\rho^2 v^i+ag^{\rho\phi}\left(E_1( \frac{\lambda}a)\partial_\rho v^i+2\frac \lambda a E_1\partial_\rho v^i-\frac\lambda a E_1 v^i+E_1(a g^{\rho\phi})E_1 v^i+a g^{\rho\phi}E_1E_1 v^i\right).
 \]
 \subsubsection{Last term}
 Here we arrive at
 \be
\nabla_{\tn_NN}v=c^3_{13} E_1(v^i)E_i.
 \ee
 \subsubsection{The full operator}
 Using \eqref{useg} we combine the above computations to obtain for a vector field that is tangential to the rescaled ellipsoids  (we use the lower index on the left to simplify the notation)
\begin{align*}
(-\dv\tilde \nabla v)^T_j&= -\nabla_i \nabla_i v^j -\frac 1{\lambda^4}v^j+(\kappa_1+\kappa_2)(\kappa_j v^j +N(v^j))\\
 &\qquad- \frac{\lambda^2}{a^2}\partial_\rho^2 v^j-ag^{\rho\phi}\left(E_1( \frac{\lambda}a)\partial_\rho v^j+2\frac \lambda a E_1\partial_\rho v^j-\frac\lambda a E_1 v^j+E_1(a g^{\rho\phi})E_1 v^j+a g^{\rho\phi}E_1E_1 v^j\right)\\
 &\qquad+v^1(c^3_{13})^2\delta^{j1}+c^3_{13} E_1(v^j),
 \end{align*}
 and using \eqref{use?} for a vector field $v$ that is divergence free both on $\R^3$ and $E$
 \begin{align*}
(-\dv\tilde \nabla v)^T_j&= -\nabla_i \nabla_i v^j -\frac 1{\lambda^4}v^j+(\kappa_1+\kappa_2)(\kappa_j v^j +N(v^j))\\
 &\qquad- \frac{\lambda^2}{a^2}\partial_\rho^2 v^j-ag^{\rho\phi}\left(E_1( \frac{\lambda}a)\partial_\rho v^j+2\frac \lambda a E_1\partial_\rho v^j-\frac\lambda a E_1 v^j+E_1(a g^{\rho\phi})E_1 v^j+a g^{\rho\phi}E_1E_1 v^j\right)\\
 &\qquad-v^1(c^3_{13})^2\delta^{j1}+c^3_{13} E_1(v^j).
 \end{align*}
 We note that the above formulas differ only by the sign of the first term in the last line.
 
  We now aim to simplify this as much as possible. At this point, we are not assuming any particular boundary condition to immediately cancel some terms.   We organize by the order of the operator, keeping the first line intact for now since it isolates the Ricci term and the Lie bracket term (which will be convenient to cancel later).  This gives 
\be\label{full}
 \begin{split}
(-\dv\tilde \nabla v)^T_j&= -\nabla_i \nabla_i v^j -\frac 1{\lambda^4}v^j+(\kappa_1+\kappa_2)(\kappa_j v^j +N(v^j))-a^2(g^{\rho\phi})^2E_1E_1 v^j\\
&\qquad+\left(c^3_{13}+\lambda \grp -ag^{\rho\phi}E_1(a g^{\rho\phi})\right)E_1v^j +v^1(c^3_{13})^2\delta^{j1}\\
 &\qquad- \frac{\lambda^2}{a^2}\partial_\rho^2 v^j-ag^{\rho\phi}\left( \frac{ag^{\rho\phi}}{\lambda^2}\partial_\rho v^j +2\frac \lambda a E_1\partial_\rho v^j \right),
 \end{split}
 \ee
 for a vector field that is tangential to the rescaled ellipsoids, and the following for a vector field that is both divergence free on $\R^3$ and $E$
 \be\label{full_divfree}
 \begin{split}
(-\dv\tilde \nabla v)^T_j&= -\nabla_i \nabla_i v^j -\frac 1{\lambda^4}v^j+(\kappa_1+\kappa_2)(\kappa_j v^j +N(v^j))-a^2(g^{\rho\phi})^2E_1E_1 v^j\\
&\qquad+\left(c^3_{13}+\lambda \grp -ag^{\rho\phi}E_1(a g^{\rho\phi})\right)E_1v^j -v^1(c^3_{13})^2\delta^{j1}\\
 &\qquad- \frac{\lambda^2}{a^2}\partial_\rho^2 v^j-ag^{\rho\phi}\left( \frac{ag^{\rho\phi}}{\lambda^2}\partial_\rho v^j +2\frac \lambda a E_1\partial_\rho v^j \right),
 \end{split}
 \ee
where we used \eqref{helpful2} and \eqref{helpful3}.  The last lines collects all the terms that depend on the behavior of $v$ away from the ellipsoid.

\subsection{Navier condition}
We are now ready to use the expansion in the scaling direction, and begin with the proof of formula \eqref{o3} for the Navier boundary condition.  Formula \eqref{o1} will then follow from it.   

Using \eqref{liebracket_gen} and \eqref{Nintrinsic}, the Navier boundary condition is equivalent to
\be
(\frac{\lambda}a\partial_\rho v^i+\frac a\lambda g^{\rho\phi}\partial_\phi v^i)E_i+sv=0,\label{nav:bc}
\ee
where $s$ is the appropriate shape operator.  Since we are rescaling in the direction of $\partial_\rho$, on the outer boundary we have
\be\label{lieb_c}
\frac 1{1+\eps}\kappa_i v^i+\frac \lambda a \partial_\rho v^i+\frac{ag^{\rho\phi}_1}{1+\eps}\frac{\partial_\phi}{\lambda} v^i=0,
\ee
where $\kappa_i$ are the principal curvatures on $E$, and where we use the notation
\[
g^{\rho\phi}_1=g^{\rho\phi}\quad \mbox{at} \ \rho=1.
\]

Now, following the discussion in Section \ref{powerseries}, we suppose
\be\label{ppo}
v=(\rho-1)^\alpha(v^\rho_\alpha\partial_\rho+ \frac 1\rho U_\alpha), \quad U_\alpha=U^\phi_\alpha\partial_\phi+U^\theta_\alpha \partial_\theta,
\ee
 
where we sum over the repeated index $\alpha$, starting at $0$, and $U_\alpha$ depend only on $\phi$ and $\theta$. 

Then using \eqref{frame_comp} the component functions of $v$ in the ON frame are
\[
v^1=\frac{g_{\phi\rho}(\rho-1)^\alpha v_\alpha^\rho+\rho^2\lambda^2 (\rho-1)^\alpha \frac 1\rho U_\alpha^\phi}{\rho\lambda},\quad v^2=a\sin\phi (\rho-1)^\alpha U_\alpha^\theta,
\]
or in a shorthand
\[
v^1=\frac{g^1_{\phi\rho}(\rho-1)^\alpha v_\alpha^\rho}{\lambda}+(\rho-1)^\alpha U_\alpha^1,\quad v^2=(\rho-1)^\alpha U_\alpha^2, 
\]
where $g_{\rho\phi}^1=g_{\rho\phi} \ \mbox{at} \ \rho=1$ and
where, for example, by $U_\alpha^2$, we mean the 2nd component of  $U_\alpha$, when written in an ON frame on the ellipsoid $E$.

For $\rho=1+\eps, j=1,2$ we have, using \eqref{usenow}
\begin{align*}
v^j&=\eps^\alpha(U_\alpha^j-\frac{a^2\grp_1}{\lambda}v_\alpha^\rho\delta^{j1}),\\
\partial_\rho v^j&=\alpha \eps^{\alpha-1}(U_\alpha^j-\frac{a^2g_1^{\phi\rho}}{\lambda}v^\rho_\alpha\delta^{j1}).
\end{align*}

We plug this into \eqref{lieb_c}, multiply by $1+\eps$, and gather the zeroth and the first power to obtain
\begin{align}
    &\eps^0:\quad \kappa_j  U_0^j+\frac \lambda a(U_1^j-\frac{a^2g^{\rho\phi}}{\lambda}v^\rho_1\delta^{j1})+ ag^{\rho\phi}E_1 U^j_0=0,\label{eps0p}\\
    &\eps^1:\quad (\kappa_j+\frac \lambda a)(U_1^j-\frac{a^2g^{\rho\phi}}{\lambda}v^\rho_1\delta^{j1} )+2\frac \lambda a (U_2^j-\frac{a^2g^{\rho\phi}}{\lambda}v^\rho_2\delta^{j1})\nonumber\\
    &\qquad\qquad+ ag^{\rho\phi}E_1(U_1^j-\frac{a^2g^{\rho\phi}}{\lambda}v^\rho_1\delta^{j1} )=0.\label{eps1p}
    \end{align}
To simplify notation, above we assume that the terms we see are evaluated at $\rho=1$. So for example, we drop the notation $g^{\rho\phi}_1$, and simply write $g^{\rho\phi}$. 

Next, on $E$ we have
\begin{align*}
 v^j&=U^j_0,\\
\partial_\rho v^j&=U_1^j-\frac{a^2g^{\rho\phi}}{\lambda}v^\rho_1\delta^{j1},\\
     \partial_\rho^2 v^j&=2(U_2^j-\frac{a^2g^{\rho\phi}}{\lambda}v^\rho_2\delta^{j1}).
     \end{align*}
We use this to evaluate \eqref{full_divfree} on $E$ together with the Navier boundary condition.  This gives

 \be\label{full_divfreeEN}
 \begin{split}
(-\dv\tilde \nabla v)^T_j&= -\nabla_i \nabla_i U^j_0 -\frac 1{\lambda^4}U^j_0-a^2(g^{\rho\phi})^2E_1E_1 U^j_0\\
&\qquad+\left(c^3_{13}+\lambda \grp -ag^{\rho\phi}E_1(a g^{\rho\phi})\right)E_1U^j_0 -U^1_0(c^3_{13})^2\delta^{j1}\\
 &\qquad- 2\frac{\lambda^2}{a^2}
 (U_2^j-\frac{a^2g^{\rho\phi}}{\lambda}v^\rho_2\delta^{j1})\\
&\qquad -ag^{\rho\phi}\left( \frac{ag^{\rho\phi}}{\lambda^2}(U_1^j-\frac{a^2g^{\rho\phi}}{\lambda}v^\rho_1\delta^{j1}) +2\frac \lambda a E_1(U_1^j-\frac{a^2g^{\rho\phi}}{\lambda}v^\rho_1\delta^{j1})\right).
 \end{split}
 \ee
  The goal is to use equations \eqref{eps0p}-\eqref{eps1p} to rewrite the above equation only in terms of $U^j_0$.  To that end, 
from \eqref{eps1p} it follows that
\[
- 2\frac{\lambda^2}{a^2}
 (U_2^j-\frac{a^2g^{\rho\phi}}{\lambda}v^\rho_2\delta^{j1})=\frac\lambda a(\kappa_j+\frac \lambda a)(U_1^j-\frac{a^2g^{\rho\phi}}{\lambda}v^\rho_1\delta^{j1} )+\lambda g^{\rho\phi}E_1(U_1^j-\frac{a^2g^{\rho\phi}}{\lambda}v^\rho_1\delta^{j1} ).
\]
We plug this into \eqref{full_divfreeEN} and observe we can combine it with the very last term in \eqref{full_divfreeEN}.  This results in
  \be
 \begin{split}
(-\dv\tilde \nabla v)^T_j&= -\nabla_i \nabla_i U^j_0 -\frac 1{\lambda^4}U^j_0-a^2(g^{\rho\phi})^2E_1E_1 U^j_0\\
&\qquad+\left(c^3_{13}+\lambda \grp -ag^{\rho\phi}E_1(a g^{\rho\phi})\right)E_1U^j_0 -U^1_0(c^3_{13})^2\delta^{j1}\\
 &\qquad +\frac\lambda a(\kappa_j+\frac \lambda a)(U_1^j-\frac{a^2g^{\rho\phi}}{\lambda}v^\rho_1\delta^{j1} )\\
&\qquad -ag^{\rho\phi}\left( \frac{ag^{\rho\phi}}{\lambda^2}(U_1^j-\frac{a^2g^{\rho\phi}}{\lambda}v^\rho_1\delta^{j1}) +\frac \lambda a E_1(U_1^j-\frac{a^2g^{\rho\phi}}{\lambda}v^\rho_1\delta^{j1})\right).
 \end{split}
 \ee
Next, from \eqref{eps0p}
\begin{align*}
 \frac \lambda a(U_1^j-\frac{a^2g^{\rho\phi}}{\lambda}v^\rho_1\delta^{j1})&=-(  \kappa_j  U_0^j+ ag^{\rho\phi}E_1 U^j_0),\\
 E_1(U_1^j-\frac{a^2g^{\rho\phi}}{\lambda}v^\rho_1\delta^{j1})&=-E_1(\frac a\lambda  \kappa_j  U_0^j+ \frac{a^2g^{\rho\phi}}{\lambda}E_1 U^j_0).
\end{align*}
We input this in the above formula to arrive at
 \be\nonumber
 \begin{split}
(-\dv\tilde \nabla v)^T_j&= -\nabla_i \nabla_i U^j_0 -\frac 1{\lambda^4}U^j_0-a^2(g^{\rho\phi})^2E_1E_1 U^j_0\\
&\qquad+\left(c^3_{13}+\lambda \grp -ag^{\rho\phi}E_1(a g^{\rho\phi})\right)E_1U^j_0 -U^1_0(c^3_{13})^2\delta^{j1}\\
 &\qquad - (\kappa_j+\frac \lambda a)(  \kappa_j  U_0^j+ ag^{\rho\phi}E_1 U^j_0)\\
&\qquad +ag^{\rho\phi}\left( \frac{ag^{\rho\phi}}{\lambda^2} (\frac a\lambda  \kappa_j  U_0^j+ \frac{a^2g^{\rho\phi}}{\lambda}E_1 U^j_0) +\frac \lambda a E_1(\frac a\lambda  \kappa_j  U_0^j+ \frac{a^2g^{\rho\phi}}{\lambda}E_1 U^j_0)\right).
 \end{split}
 \ee
From here on we organize by the order of derivatives.  We observe that the double $E_1$ derivatives cancel one another.  Now we discuss the first order term.  

 \subsubsection{First order term}
We gather the coefficients of $E_1U^j_0$.  They are
\begin{align*}
&c^3_{13}+\lambda \grp -ag^{\rho\phi}E_1(a g^{\rho\phi})- (\kappa_j+\frac \lambda a)ag^{\rho\phi} +ag^{\rho\phi} \frac{ag^{\rho\phi}}{\lambda^2}  \frac{a^2g^{\rho\phi}}{\lambda}+a\grp \kappa_j  +\lambda\grp E_1 (\frac{a^2g^{\rho\phi}}{\lambda}) \\
&=c^3_{13} +ag^{\rho\phi} \frac{ag^{\rho\phi}}{\lambda^2}  \frac{a^2g^{\rho\phi}}{\lambda}  -\frac{a^2(g^{\rho\phi})^2}{\lambda} E_1 (\lambda) ,
\end{align*}
where we use the product rule and make some other obvious cancellations.  From \eqref{helpful3} we see we are left only with $c^3_{13}$, which comes from
\[
\nabla_{\tn_NN}v=c^3_{13}\nabla_{E_1}U_0=c^3_{13}E_1U^j_0E_j.
\]
Recall $c_{13}^3$ is the normal component of $[E_1,N]$. It can also be interpreted as $\frac{E_1(|\nabla \rho|)}{|\nabla \rho|}$.
\subsubsection{Zero order term}
The coefficients of $U^j_0$ are
\[
-\frac 1{\lambda^4}-(c^3_{13})^2\delta^{j1} - (\kappa_j+\frac \lambda a)  \kappa_j +ag^{\rho\phi}\frac{a^2g^{\rho\phi}}{\lambda^3}\kappa_j+\lambda g^{\rho\phi} E_1(\frac a\lambda  \kappa_j ).
\]
We have
\[
\kappa_1=-\frac{a}{\lambda^3},\quad \kappa_2=-\frac1{a\lambda}, 
\]
so
\be
E_1(\frac a\lambda \kappa_1)=4\frac{a^4}{\lambda^7}g^{\rho\phi},\quad E_1(\frac a\lambda \kappa_2)=2\frac{a^2}{\lambda^5}g^{\rho\phi}.
\ee
Using
\[
c^3_{13}=\frac{a^2\grp}{\lambda^3},
\]
for $j=1$, we have
\[
-\frac 1{\lambda^4}+2(c^3_{13})^2-(\kappa_1)^2-\frac \lambda a \kappa_1,
\]
which by \eqref{helpful1} reduces to
\[
\frac 1{\lambda^4}-3(\kappa_1)^2-\frac{1}{\lambda^2}+2\frac{a^2}{\lambda^4}.
\]
  
For $j=2$, we have the following terms
\[
-\frac 1{\lambda^4}-(\kappa_2)^2+\frac 1{a^2}+\frac{a^2(g^{\rho\phi})^2}{\lambda^4},
\]
which using \eqref{helpful1} again simplify to
  \[
 \frac{1}{\lambda^2}-\frac{2}{\lambda^4},
 \]
This can be interpreted as
\[
-\frac{\kappa_2}{\ag}-2K_E,
\]
or more intrinsically as
\be\label{0u2}
\sqrt{K_E}-2K_E.
\ee
\subsubsection{Summary}\label{summary}
 We summarize the result here.  Under the expansion \eqref{ppo} and the homogenous Navier boundary condition, the heuristic of the thin limit produces the following operator.
 \begin{equation}\nonumber
\begin{split}
  &\nabla^\ast \nabla U_0+
  c^3_{13}E_1U_0^jE_j \\
 &\quad+(K_E-\sqrt{K_E}-3(\kappa_1)^2-2\frac{\kappa_1}{\abs{\nabla \rho}})U^1_0 E_1\\
 &\quad+(\sqrt{K_E}-2K_E)U^2_0E_2.
\end{split}
\end{equation}
Moreover, using \eqref{helpful5}, we can write it as
 \begin{equation}\label{summaryN}
\begin{split}
  &\nabla^\ast \nabla U_0+c^3_{13}\nabla_{E_1}U_0\\
 &\quad+(K_E-\sqrt{K_E}-3(\kappa_1)^2-2\frac{\kappa_1}{\abs{\nabla \rho}})U^1_0 E_1\\
 &\quad+(\sqrt{K_E}-2K_E)U^2_0E_2.
\end{split}
\end{equation}
To recover a fully intrinsic formula we use the following lemma.

\begin{lemma}\label{Key1}
Let $w \in \mathfrak X(E)$ and let $\{E_1 , E_2 \}$ be the orthonormal moving frame on $E$ as given by $E_1 = \frac{\partial_{\phi}}{\lambda}$ and $E_2 = \frac{\partial_{\theta}}{a \sin \phi}$. Then, the following identities hold on $E$

\begin{equation}\label{KeyandImportant}
\Big ( g(w, \nabla_{\cdot } (c_{13}^3 E_1 )) \Big )^{\sharp} = \nabla_w \big ( c_{13}^3 E_1 \big ) ,
\end{equation}
\begin{equation}\label{observation1}
\bs
-  \nabla_w \big ( c_{13}^3 E_1 \big )
&= \Ric w + \Big ( K_E - \sqrt{K_E} - 3 \big ( \kappa_1 \big )^2 -2 \frac{\kappa_1}{\big | \nabla \rho \big | } \Big ) w^1E_1\\
 &\quad+  \big ( \sqrt{K_E} - 2 K_E \big )w^2 E_2.
 \end{split}
\end{equation}
 
\end{lemma}
\begin{proof} First, for an arbitrary vector field $w = w^1 E_1 + w^2 E_2$, by using $\nabla_{E_1}E_1 = 0$ and $E_2 \big ( c_{13}^3 \big ) = 0$, we get by a direct computation
\begin{equation}\label{morefavorable}
\nabla_w \big ( c_{13}^3 E_1 \big ) = w^1 E_1 \big ( c_{13}^3 \big ) E_1 + w^2 c_{13}^3 \nabla_{E_2} E_1 .
\end{equation}
On the other hand, by using in addition that \eqref{helpful5} gives $g( E_1 , \nabla_{E_2} E_1) = 0$, it follows
\begin{equation}\label{lessfavorable}
\begin{split}
\Big ( g_E (w, \nabla_{\cdot } (c_{13}^3 E_1 )) \Big )^{\sharp} & = \sum_{i=1}^2 g(w, \nabla_{E_i } (c_{13}^3 E_1 )) E_i \\
&=\sum_{i=1}^2 g(w, E_i  (c_{13}^3) E_1+c_{13}^3\nabla_{E_i}E_1) E_i \\
%& = w^1E_1 \big ( c_{13}^3 \big ) E_1 + g(w , c_{13}^3 \nabla_{E_2} E_1  ) E_2 \\
& = w^1 E_1 \big ( c_{13}^3 \big ) E_1 + w^2 c_{13}^3  g( E_2 , \nabla_{E_2} E_1  ) E_2 \\
& = w^1 E_1 \big ( c_{13}^3 \big ) E_1 + w^2 c_{13}^3 \nabla_{E_2} E_1 ,
\end{split}
\end{equation}
where the last equal sign follows since $g( E_1 , \nabla_{E_2} E_1) = 0$ means that $\nabla_{E_2} E_1 = g( E_2 , \nabla_{E_2} E_1  ) E_2 $.  This shows \eqref{KeyandImportant}.

Next, we start by recalling \eqref{helpful4}
\begin{equation}\label{calculus1}
\begin{split}
E_1 \big ( c_{13}^3 \big |_E \big ) = - \frac{2}{\lambda^4} (1 + a^2) + \frac{1}{\lambda^2} + \frac{3a^2}{\lambda^6} .
\end{split}
\end{equation}
as well as \eqref{helpful5}
\begin{equation}\label{appearbefore}
  \nabla_{E_2}E_1 =  \frac{\cos \phi}{\lambda\sin \phi} E_2 .
\end{equation}
Then from \eqref{morefavorable}, \eqref{c313} and \eqref{helpful1} we get

\be\label{calculus3}
-\nabla_w \big ( c_{13}^3 E_1 \big ) = w^1 \Big ( \frac{2}{\lambda^4} (1+a^2) - \frac{1}{\lambda^2} - \frac{3a^2}{\lambda^6} \Big ) E_1 - w^2 \frac{(1-\lambda^2)}{\lambda^4} E_2.
\ee

By taking $\big ( \kappa_1 \big )^2 = \frac{a^2}{\lambda^6}$, $K_E = \frac{1}{\lambda^4}$, and $\frac{\kappa_1}{\big | \nabla \rho \big |}
= - \frac{a^2}{\lambda^4}$ into our consideration, \eqref{calculus3} immediately leads to
\begin{align*}
 -\nabla_w \big ( c_{13}^3 E_1 \big ) 
= & \big ( K_E w^1 E_1 + K_E w^2 E_2 \big ) +w^1 \Big ( K_E - \sqrt{K_E} - 3 \big ( \kappa_1 \big )^2 -2 \frac{\kappa_1}{\big | \nabla \rho \big | } \Big ) E_1 \\
&\quad+ w^2 \big ( \sqrt{K_E} - 2 K_E \big ) E_2 \\
= & \Ric w + w^1 \Big ( K_E - \sqrt{K_E} - 3 \big ( \kappa_1 \big )^2 -2 \frac{\kappa_1}{\big | \nabla \rho \big | } \Big ) E_1 + w^2 \big ( \sqrt{K_E} - 2 K_E \big ) E_2 ,
\end{align*}
which is exactly identity \eqref{observation1} .
\end{proof}
In light of Lemma \ref{Key1}, our thin shell limit result can be interpreted as follows.
\begin{equation*}
\begin{split}
& \nabla^* \nabla U_0 + c_{13}^3 \nabla_{E_1} U_0
+  \Big ( K_E - \sqrt{K_E} - 3 \big ( \kappa_1 \big )^2 -2 \frac{\kappa_1}{\big | \nabla \rho \big | } \Big ) U_0^1 E_1 +  \big ( \sqrt{K_E} - 2 K_E \big ) U_0^2 E_2 \\
&= \nabla^* \nabla U_0 +c_{13}^3 \nabla_{E_1} U_0 - \Ric U_0 -\nabla_{U_0} \big ( c_{13}^3 E_1 \big ) \\
&=-2\dv \Def U_0+[c_{13}^3E_1, U_0],
\end{split}
\end{equation*}
where we use \eqref{BW}.  The last line is the formula \eqref{o3} in Theorem \ref{thmd}.  

From here, since the formulas \eqref{full_divfree} and \eqref{full} differ only by the sign of the term $(c^3_{13})^2U^1_0E_1$, and the asymptotic expansion of the vector field $v$ in Theorem \ref{thm1} is a special case of the asymptotic expansion we use in the proof of Theorem \ref{thmd}, we can obtain the formula \eqref{o1} by adding $2(c^3_{13})^2U^1_0E_1$ to the formula \eqref{o3}.

\subsection{Hodge condition}
From Proposition \ref{NavierLie}, Proposition \ref{HodgeLie}, and \eqref{NH}, we have that the homogeneous Hodge boundary condition is equivalent to
\[
[N,v]^T=2sv.
\]
By definition of the Lie bracket and the Weingarten equation, this is equivalent to
\[
(\tn_N v)^T-sv=0.
\]
Hence, using \eqref{liebracket_gen} and \eqref{Nintrinsic}, the homogeneous Hodge boundary condition is equivalent to
\[
(\frac{\lambda}a\partial_\rho v^i+\frac a\lambda g^{\rho\phi}\partial_\phi v^i)E_i-sv=0.
\]
We note that this only differs from the homogeneous Navier condition by the sign of the shape operator term.

Then, similarly as before, on the outer boundary we have
\be\label{lieb_c_hodge}
\frac \lambda a \partial_\rho v^i+\frac{ag^{\rho\phi}_1}{1+\eps}\frac{\partial_\phi}{\lambda} v^i-\frac 1{1+\eps}\kappa_i v^i=0,
\ee
 and
 \begin{align*}
v^j&=\eps^\alpha(U_\alpha^j-\frac{a^2\grp_1}{\lambda}v_\alpha^\rho\delta^{j1}),\\
\partial_\rho v^j&=\alpha \eps^{\alpha-1}(U_\alpha^j-\frac{a^2g_1^{\phi\rho}}{\lambda}v^\rho_\alpha\delta^{j1}).
\end{align*}

Using the same set-up as in the previous section, including assuming that the terms we see are evaluated at $\rho=1$, it follows 
\begin{align}
    &\eps^0:\quad -\kappa_j  U_0^j+\frac \lambda a(U_1^j-\frac{a^2g^{\rho\phi}}{\lambda}v^\rho_1\delta^{j1})+ ag^{\rho\phi}E_1 U^j_0=0,\label{eps0p:hodge}\\
    &\eps^1:\quad (\frac \lambda a-\kappa_j)(U_1^j-\frac{a^2g^{\rho\phi}}{\lambda}v^\rho_1\delta^{j1} )+2\frac \lambda a (U_2^j-\frac{a^2g^{\rho\phi}}{\lambda}v^\rho_2\delta^{j1})\nonumber\\
    &\qquad\qquad+ ag^{\rho\phi}E_1(U_1^j-\frac{a^2g^{\rho\phi}}{\lambda}v^\rho_1\delta^{j1} )=0.\label{eps1p:hodge}
    \end{align}

Next, from \eqref{full_divfree} on $E$ together with the Hodge boundary condition, we obtain

 \be\label{full_divfreeEH}
 \begin{split}
(-\dv\tilde \nabla v)^T_j&= -\nabla_i \nabla_i U^j_0 -\frac 1{\lambda^4}U^j_0+2(\kappa_1+\kappa_2)\kappa_iU_0^i-a^2(g^{\rho\phi})^2E_1E_1 U^j_0\\
&\qquad+\left(c^3_{13}+\lambda \grp -ag^{\rho\phi}E_1(a g^{\rho\phi})\right)E_1U^j_0 -U^1_0(c^3_{13})^2\delta^{j1}\\
 &\qquad- 2\frac{\lambda^2}{a^2}
 (U_2^j-\frac{a^2g^{\rho\phi}}{\lambda}v^\rho_2\delta^{j1})\\
&\qquad -ag^{\rho\phi}\left( \frac{ag^{\rho\phi}}{\lambda^2}(U_1^j-\frac{a^2g^{\rho\phi}}{\lambda}v^\rho_1\delta^{j1}) +2\frac \lambda a E_1(U_1^j-\frac{a^2g^{\rho\phi}}{\lambda}v^\rho_1\delta^{j1})\right).
 \end{split}
 \ee
    
Then, following similar calculations as for the Navier boundary condition we can arrive at the following operator: 
 \begin{equation}\label{summaryH}
\begin{split}
  &\nabla^\ast \nabla U_0+c^3_{13}\nabla_{E_1}U_0\\
 &\quad+(\frac{3}{\lambda^2}-\frac{3+4a^2}{\lambda^4}+\frac{5a^2}{\lambda^6})U^1_0 E_1\\
 &\quad+(\frac{2}{\lambda^4}-\frac{1}{\lambda^2})U^2_0E_2,
\end{split}
\end{equation}
which can be interpreted as 
 \begin{equation}\label{summaryH1}
\begin{split}
  &\nabla^\ast \nabla U_0+c^3_{13}\nabla_{E_1}U_0\\
 &\quad+(3\sqrt{K_E}-3K_E+5(\kappa_1)^2+4\frac{\kappa_1}{|\nabla\rho|}
 )U^1_0 E_1\\
 &\quad+(2K_E-\sqrt{K_E})U^2_0E_2.
\end{split}
\end{equation}
Then from Lemma \ref{Key1} and \eqref{c313squared} we obtain

\begin{equation*}
\begin{split}
& \nabla^* \nabla U_0 + c_{13}^3 \nabla_{E_1} U_0
+  \Big ( 3 \sqrt{K_E} -3K_E +5 \big ( \kappa_1 \big )^2 +4 \frac{\kappa_1}{\big | \nabla \rho \big | } \Big ) U_0^1 E_1 +  \big (  2 K_E -\sqrt{K_E}\big ) U_0^2 E_2 \\
=& \nabla^* \nabla U_0 + c_{13}^3 \nabla_{E_1} U_0 +  \Big ( g( U_0 , \nabla_{\cdot} \big ( c_{13}^3 E_1  \big ) ) \Big )^{\sharp}+ \Ric U_0- 2 \big ( c_{13}^3 \big )^2 U_0^1 E_1\\
=&( \dd^\ast \dd U_0^\flat+\mathcal L_{c_{13}^3 {E_1}} U_0^\flat- 2 \big ( c_{13}^3 \big )^2 U_0^1 E^1)^\sharp,
\end{split}
\end{equation*}
where we use  that the $i$'th component function of the Lie derivative $\mathcal L_Xu^\flat$ in the frame $E_i$ can be written as (see for example \cite{CC25})
\be
(\mathcal L_Xu^\flat)_i=g(\nabla_X u, E_i)+g(u, \nabla_{E_i}X).
\ee
The last line above is exactly formula \eqref{o4} when apply the $\flat$ operator.  Then similarly as for the Navier boundary condition, we can obtain formula \eqref{o2} by adding $2 \big ( c_{13}^3 \big )^2 U_0^1 E_1$.  This completes the proof of Theorems \ref{thm1} and \ref{thmd}.
 
 \section{Proof of Theorem \ref{thm2}: expansion in the normal direction with the Hodge condition}\label{sectHodge}

 We would like to perform an asymptotic expansion in the normal direction.  We let
\be\label{exp:sigma}
v=\sigma^\alpha U_\alpha,
\ee
where $\sigma$ is the (signed) distance from the ellipsoid, and each $U_\alpha$ is extended by parallel transport in the normal direction.  The normal $N$ is also extended by parallel transport in the normal direction, and denoted still by $N$.  In particular, we have  
\[
\tn_N U_\alpha= 0, \quad \tn_N N=0,
\]
for all points in the tubular neighborhod $\mathcal T(\eps_0)$.

To plug into \eqref{project}, we compute the following term.
\[
\tn_N\tn_N(\sigma^\alpha U_\alpha)=\tn_N (N(\sigma^\alpha)U_\alpha+\sigma^\alpha \tn_N U_\alpha)=N(N\sigma^\alpha)U_\alpha,
\]
 by parallel transport condition.  Next for each $\alpha$
 \[
 N\sigma^\alpha=\alpha \sigma^{\alpha-1}N\sigma,
 \]
 so
 \[
N N\sigma^\alpha=\alpha N(\sigma^{\alpha-1}N\sigma)=\alpha(\alpha-1)\sigma^{\alpha-2}(N\sigma)^2+\alpha \sigma^{\alpha-1}NN\sigma.
 \]
By \eqref{Npt}
 \[
 N\sigma=\partial_\sigma \sigma=1,
 \]
 so $NN\sigma=0$, and we get
 \[
N N\sigma^\alpha=\alpha(\alpha-1)\sigma^{\alpha-2},
 \]
 and
 \be\label{eq1}
 (\tn_N\tn_N(\sigma^\alpha U_\alpha))^T=2U_2.
 \ee

Then, from \eqref{project}, it follows
 
 \be\label{projectE11}
(-\dv\tilde \nabla v)^T=\nabla^\ast  \nabla U_0-\Ric U_0+nH[N,v]^T-2U_2.
\ee
 Next, we observe that if $U_\alpha$ is a vector field on $E$, 
\[
U_\alpha=U_\alpha^{\lp} \partial_{\lp}+U_\alpha^{\lt} \partial_{\lt},
\]
then its extension by parallel transport in the direction $N=\partial_{\sigma}$ is
\[
U_{\alpha}=\frac{U_\alpha^{\lp}}{1+\sigma \tfrac{a}{\lambda^3}}\partial_{\lp}+\frac{U_\alpha^{\lt}}{1+ \tfrac{\sigma}{a\lambda}}\partial_{\lt}.
\]

Now we calculate the relations given by the Hodge boundary condition in terms of $v^\flat$ at distance $\sigma$ of the ellipsoid. The condition reads
\[
\iota^*_{\tilde E_\sigma}(\mathcal{L}_Nv^\flat)=0,
\]
where
\begin{equation}\label{Esigma}
\tilde E_{\sigma} = \big \{ p + \sigma N_p : p \in E     \big \}.
\end{equation}
We begin calculating  
\begin{align*}
v^\flat&= g_{\tilde\phi\tilde\phi} v^{\tilde\phi} d\tilde\phi+g_{\tilde\theta\tilde\theta}v^{\tilde\theta} d\tilde\theta\\
&= \sigma^\alpha \frac{U^{\tilde\phi}_\alpha}{(1-\sigma \kappa_1)}g_{\tilde\phi\tilde\phi}d\tilde\phi + \sigma^\alpha \frac{U^{\tilde\theta}_\alpha}{(1-\sigma \kappa_2)}g_{\tilde\theta\tilde\theta}d\tilde\theta\\
&= \sigma^\alpha U^{\tilde\phi}_\alpha(1-\sigma \kappa_1)\lambda^2 d\tilde\phi + \sigma^\alpha U^{\tilde\theta}_\alpha(1-\sigma \kappa_2)a^2 \sin^2\tilde\phi d\tilde\theta
\end{align*}
Notice that in this case, $\iota_N v^\flat=0$, so in particular using the Cartan's formula we have
\[
\mathcal{L}_N v^\flat=\iota_N dv^\flat.
\]
We now calculate 
\begin{align*}
\iota_N dv^\flat&= \partial_\sigma(\sigma^\alpha U^{\tilde\phi}_\alpha(1-\sigma \kappa_1)\lambda^2) d\tilde\phi  +\partial_\sigma(\sigma^\alpha U^{\tilde\theta}_\alpha(1-\sigma \kappa_2)a^2 \sin^2\tilde\phi) d\tilde\theta \\
&= [\alpha \sigma^{\alpha -1} U_\alpha^{\tilde\phi}(1-\sigma\kappa_1)\lambda^2-\sigma^\alpha U_\alpha^{\tilde\phi}\kappa_1\lambda^2]d\tilde\phi 
+ [\alpha \sigma^{\alpha -1} U_\alpha^{\tilde\theta}(1-\sigma\kappa_2)a^2\sin^2\tilde\phi-\sigma^\alpha U_\alpha^{\tilde\theta}\kappa_2 a^2\sin^2\tilde\phi]d\tilde\theta.
\end{align*}

Gathering the powers of $\sigma$, we conclude that the Hodge boundary condition at distance $\sigma$ gives us the following relations
\begin{align}
&   \sigma^0: U^{\tilde\phi}_1=\kappa_1 U^{\tilde\phi}_0, \, U_1^{\tilde\theta}=\kappa_2 U_0^{\tilde\theta}\label{bc:line1} \\
&\sigma^1: U^{\tilde\phi}_2=(\kappa_1)^2 U^{\tilde\phi}_0, \, U^{\tilde\theta}_2=(\kappa_2)^2 U_0^{\tilde\theta}.\label{bc:line2}
\end{align}

We now proceed to simplify formula \eqref{projectE11}. First, notice that with the Navier boundary condition and \eqref{NH}, we have $[N,v]^T=2sv$.

So the formula \eqref{projectE11} becomes 
\be
(-\dv\tilde \nabla v)^T=\nabla^\ast  \nabla U_0-\Ric U_0+2(\kappa_1+\kappa_2)sU_0-2U_2.
\ee

Using \eqref{bc:line2} we obtain
\begin{align*}
(-\dv\tilde \nabla v)^T_j&=\nabla_i \nabla_i U_0^j-(\kappa_1\kappa_2) U_0^j+2(\kappa_1+\kappa_2)\kappa_j U_0^j -2(\kappa_j)^2U_0^j\\
&=\nabla_i \nabla_i U_0^j+(\kappa_1\kappa_2) U_0^j,
\end{align*}
which by \eqref{BW} is exactly the Hodge Laplacian.

\section{Comparison to previous works}\label{comp}

\subsection{Expansion in the normal direction, by parallel transport: Navier condition}
Here we would like to recover Miura's heuristic result \cite{Miura}, but using the formula \eqref{project} we have been working with.  

In this case, the beginning is the same as in the Hodge condition, and using that the Navier condition is equivalent to the Lie bracket condition, from \eqref{project} and \eqref{eq1}, we can start with
\be\label{projectE1}
(-\dv\tilde \nabla v)^T=\nabla^\ast  \nabla v-\Ric v-2U_2.
\ee

Next, the boundary condition distance $\sigma$ away gives
\[
0=[N,v]^T=[\partial_{\sigma}, v]^T=[\partial_{\sigma}, \sigma^\alpha U_\alpha]^T=\sigma^\alpha[\partial_{\sigma},  U_\alpha]^T+\alpha \sigma^{\alpha-1}U_\alpha^T,
\]
and for each $\alpha$
\[
[\partial_{\sigma},  U_\alpha]^T=\partial_\sigma\Big(\frac{U_\alpha^{\lp}}{1+\sigma \tfrac{a}{\lambda^3}}\Big)\partial_{\lp}+\partial_\sigma\Big(\frac{U_\alpha^{\lt}}{1+ \tfrac{\sigma}{a\lambda}}\Big)\partial_{\lt}=
-\frac{U_\alpha^{\lp}}{(1+\sigma \tfrac{a}{\lambda^3})^2}\frac a{\lambda^3}\partial_{\lp}-\frac{U_\alpha^{\lt}}{(1+ \tfrac{\sigma}{a\lambda})^2}\frac 1{a\lambda}\partial_{\lt},
\]
which we recognize as
\[
[\partial_{\sigma},  U_\alpha]^T= \frac{U_\alpha^{\lp}}{(1-\sigma \kappa_1)^2}\kappa_1\partial_{\lp}+\frac{U_\alpha^{\lt}}{(1-\sigma \kappa_2)^2}\kappa_2\partial_{\lt},
\]
where $\kappa_1, \kappa_2$ are the principal curvatures of $E$ at $p$.  It follows
\[
0=[N,v]^T= \sigma^\alpha \Big(\frac{U_\alpha^{\lp}}{(1-\sigma \kappa_1)^2}\kappa_1\partial_{\lp}+\frac{U_\alpha^{\lt}}{(1-\sigma \kappa_2)^2}\kappa_2\partial_{\lt}\Big)+\alpha \sigma^{\alpha-1}U_\alpha.
\]
Multiplying by $(1-\sigma\kappa_1)^2(1-\sigma\kappa_2)^2$ and collecting powers of $\sigma$ we deduce

\begin{align}
&   \sigma^0:\quad \kappa_1U_0^{\lp}+U_1^{\lp}=0, \quad \kappa_2U_0^{\lt}+U_1^{\lt}=0,\label{1sline}\\
&\sigma^1:\quad U_2^{\lp}=\kappa_2(\kappa_1U_0^{\lp}+U_1^{\lp}),\quad U_2^{\lt}=\kappa_1(\kappa_2U_0^{\lt}+U_1^{\lt}).\label{2ndline}
\end{align}
Plugging in \eqref{1sline} into \eqref{2ndline}, we have $U_2=0$, and \eqref{projectE11} reduces to the deformation Laplacian, which is consistent with what was obtained by Miura in \cite{Miura}.

\appendix
\section{Power series expansion: justification}\label{justify}
We start with the following definition.

\begin{defn}
Let $E = \{ (x,y,z) \in \mathbb{R}^3 : \frac{x^2}{a^2} + \frac{y^2}{a^2} + z^2 = 1   \}$, and let $\Psi = (\Psi^1 , \Psi^2 , \Psi^3 ) : D \rightarrow E$ be a local parametrization of
$E$, with $D$ be an open set in $\mathbb{R}^2$. We say that $\Psi : D \rightarrow E$ is a \emph{real analytic} if the three component functions $\Psi^1 , \Psi^2 , \Psi^3$ are real analytic functions on $D$.
\end{defn}

\begin{remark}\label{RemarkC1}
Here, we would like to note that: \emph{each point $p_0 \in E$ can be covered by at least one real analytic local parametrization $\Psi : D \rightarrow E$ so that $p_0 \in \Psi (D)$.} 
\end{remark}

\begin{defn}
Let $v = (v^1 , v^2 , v^3) = v^1 \partial_x + v^2 \partial_y + v^3 \partial_y$ be a smooth vector field defined in some open neighborhood $\mathcal{N}$ of the ellipsoid $E$ in $\mathbb{R}^3$. We say that $v$ is a real analytic vector field on $\mathcal{N}$ if all the three component functions  $v^1 , v^2 , v^3$ are real analytic functions on $\mathcal{N}$.
\end{defn}

The goal of this appendix is to justify a power series expansion for a real analytic vector field in a neighborhood of the ellipsoid $E$.  We accomplish this in two steps.  First, we consider a strictly tangential vector field to the rescaled ellipsoid $E$.  Then we upgrade the result to cover the case of a general vector field, which means having a nonzero $\partial_\rho$ component.

\begin{lemma}\label{basic}
Let $E = \{ (x,y,z) \in \mathbb{R}^3 : \frac{x^2}{a^2} + \frac{y^2}{a^2} + z^2 = 1   \}$, and let $\eps_0>0$. Consider the open neighborhood $\mathcal{N}(\eps_0)$ of $E$ given by
\begin{equation}\label{rescaledNeighbour}
\mathcal{N}(\eps_0) = \Big \{ (x,y,z) \in \mathbb{R}^3 : 1-\eps_0 < \Big ( \frac{x^2}{a^2} + \frac{y^2}{a^2} + z^2 \Big )^{\frac{1}{2}} < 1 + \eps_0  \Big \}.
\end{equation}
Let $v$ be a real analytic vector field on $\mathcal{N}(\eps_0)$ which is tangential to every single rescaled ellipsoid $\{ (\rho x , \rho y , \rho z ) : (x,y,z) \in E \}$, for all $\rho \in (1-\eps_0 , 1 + \eps_0 )$. Then, it follows that there exists some $\widetilde{\eps} \in (0 ,\eps_0 )$, together with a uniquely determined sequence
$\{\tilde U_\alpha\}_{\alpha = 0}^{\infty}$ of real analytic vector fields on $\mathcal{N}(\widetilde{\eps})$
which satisfies the following properties:
\begin{itemize}
\item For each integer $\alpha \geq 0$, $\tilde U_{\alpha}$ is parallel along every single straight line radiating out from the origin $O = (0,0,0)$.
\item For each integer $\alpha \geq 0$, $\tilde U_{\alpha}$ is tangential to every single rescaled ellipsoid $\{ (\rho x , \rho y , \rho z ) : (x,y,z) \in E \}$, for all $\rho \in (1- \widetilde{\eps} , 1 + \widetilde{\eps} )$.
\item $v = \sum_{\alpha =0}^{\infty} (\rho - 1)^{\alpha}\tilde U_{\alpha}$ holds on $\mathcal{N}(\widetilde{\eps})$ .
\end{itemize}
\end{lemma}
\begin{proof}
\textbf{Step 1.}
 For any open subset $U$ of the ellipsoid $E$, we define  $$\mathcal{C}(U) = \{ (\rho x , \rho y , \rho z ) : \rho > 0 , (x,y,z) \in U \},$$ 
 which is an open set in $\mathbb{R}^3$. Let $\Psi = (\Psi^1 , \Psi^2 , \Psi^3 ) : D \rightarrow E$ be a real analytic local parametrization of $E$, so that $\Psi (D)$ is some open set in $E$. Then, we can consider the associated real analytic
local parametrization $\widetilde{\Psi} : (0,\infty ) \times D \rightarrow \mathcal{C}(\Psi(D))$ given by
\begin{equation}\label{associatedPar}
\widetilde{\Psi} ( \rho , \tau^1 , \tau^2 ) = ( \rho \Psi^1 ( \tau^1 , \tau^2 ) , \rho \Psi^2 ( \tau^1 , \tau^2 ) , \rho \Psi^3 ( \tau^1 , \tau^2 )         ) .
\end{equation}
In accordance with the real analytic version of the inverse mapping theorem, the real analyticity of $\widetilde{\Psi}$ implies that the inverse map
$$ \widetilde{\Psi}^{-1} = ( \rho , \tau^1 , \tau^2 ) : \mathcal{C}(\Psi(D)) \rightarrow (0 , \infty ) \times D $$ is also real analytic. Thus, the three associated coordinate functions $\rho , \tau^1 , \tau^2$ are all real analytic on $\mathcal{C}(\Psi(D))$.  

\textbf{Step 2.} Consider a real analytic vector field $v = (v^1 , v^2 , v^3) = v^1 \partial_x + v^2 \partial_y + v^3 \partial_y$ on $\mathcal{N}(\eps_0)$ given by 
\eqref{rescaledNeighbour},
which is
tangential to all the rescaled ellipsoids $\{ (\rho x , \rho y , \rho z) : (x,y,z) \in E   \} $, for any $\rho \in (1-\eps_0 , 1 + \eps_0)$ .

Take an arbitrary point $p_0 = (x_0, y_0 , z_0) \in E$. By Remark \ref{RemarkC1}, we can choose a
real analytic local parameterization $\Psi : D \rightarrow E$ covering $p_0$ so that $p_0 \in \Psi (D)$. Let
$\widetilde{\Psi} : (1-\eps_0 , 1 + \eps_0 )\times D \rightarrow \mathcal{C}(\Psi (D))\cap \mathcal{N}(\eps_0) $ be the associated real analytic
local parametrization as given in \eqref{associatedPar}. As said in \textbf{Step 1}, the three coordinate functions $\rho , \tau^1 ,\tau^2 $ are real analytic on $\mathcal{C}(\Psi (D))\cap \mathcal{N}(\eps_0) $.
We note that the coordinate function $\rho$ coincides with $\big ( \frac{x^2}{a^2} + \frac{y^2}{a^2} + z^2 \big )^{\frac{1}{2}}$.

Under
this real analytic coordinate system $(\rho, \tau^1 ,\tau^2 ) = \widetilde{\Psi}^{-1}$, we can express $v$ locally on  $\mathcal{C}(\Psi (D))\cap \mathcal{N}(\eps_0) $ as follows
\begin{equation}\label{localdecomp}
v = v^{\tau^1} \partial_{\tau^1} + v^{\tau^2} \partial_{\tau^2} .
\end{equation}
For each $k=1,2$, we have for any $p \in \mathcal{C}(\Psi (D))\cap \mathcal{N}(\eps_0) $ that
\begin{equation}\label{straightforward}
\partial_{\tau^k} \Big |_{p} = ( \rho (p) \partial_{\tau^k} \Psi^{1} (\widetilde{\Psi}^{-1}(p))   , \rho (p) \partial_{\tau^k} \Psi^{2} (\widetilde{\Psi}^{-1}(p)) , \rho (p) \partial_{\tau^k} \Psi^{3} (\widetilde{\Psi}^{-1}(p))         ).
\end{equation}
Since $\rho$, $\Psi^{k}$, and $\widetilde{\Psi}^{-1}$ are all real analytic, it is apparent from \eqref{straightforward} that the vector field $\partial_{\tau^k}$ is real analytic on $\mathcal{C}(\Psi (D))\cap \mathcal{N}(\eps_0)$. Thus, it follows that $v^{\tau^1}$ and $v^{\tau^2}$ are real analytic on $\mathcal{C}(\Psi (D))\cap \mathcal{N}(\eps_0)$. 

Hence, $v^{\tau^1}\circ \widetilde{\Psi}$ and $v^{\tau^1}\circ \widetilde{\Psi}$ are real analytic functions on $(1-\eps_0 ,1 + \eps_0) \times D$. We write
$\widetilde{\Psi}^{-1} (p_0) = (1, \tau_0^1 , \tau_0^2)$. Then, we can express $\rho v^{\tau^k}\circ \widetilde{\Psi}$ as a power series expansion in terms of powers of $(\rho -1 )$, powers of $(\tau^1 - \tau_0^1)$ and powers of $(\tau^2 - \tau)^2$ as follows
\begin{equation}\label{expan1}
\begin{split}
\rho \big( v^{\tau^k}\circ \widetilde{\Psi} \big ) (\rho, \tau^1 , \tau^2 )
&=  \sum_{\alpha=0}^{\infty} \sum_{a = 0}^{\infty} \sum_{b=0}^{\infty} v_{\alpha , a , b}^k \big ( \tau^1 - \tau_0^1  \big )^a  \big ( \tau^2 - \tau_0^2  \big )^b  (\rho - 1)^{\alpha}    \\
&=  \sum_{\alpha=0}^{\infty} (\rho - 1)^{\alpha} v_{\alpha}^{k} (\tau^1 , \tau^2 ) ,
\end{split}
\end{equation}
which is valid for all $(\rho , \tau^1 , \tau^2) \in (1-\eps , 1 + \eps ) \times (\tau_0^1 -\eps , \tau_0^1 + \eps) \times (\tau_0^2 -\eps , \tau_0^2 + \eps)$, with some sufficiently small $\eps \in (0, \eps_0)$.
In \eqref{expan1}, $v_{\alpha, a, b}^k \in \mathbb{R}^1$, for $\alpha , a , b \in \mathbb{N}$, and
\begin{equation}
v_{\alpha}^{k} (\tau^1 , \tau^2 ) = \sum_{a = 0}^{\infty} \sum_{b=0}^{\infty} v_{\alpha , a , b}^k \big ( \tau^1 - \tau_0^1  \big )^a  \big ( \tau^2 - \tau_0^2  \big )^b.
\end{equation}

Consider the associated open set
\begin{equation}
U(p_0) = \widetilde{\Psi} ( (1-\eps , 1 + \eps ) \times (\tau_0^1 -\eps , \tau_0^1 + \eps) \times (\tau_0^2 -\eps , \tau_0^2 + \eps)     ).
\end{equation}
By
\eqref{expan1}, we can now express $v$ locally on $U(p_0)$ as
\begin{equation}\label{PowerSeriesExapn}
\begin{split}
v & = \sum_{k=1}^2 \rho v^{\tau^k} \frac{1}{\rho} \partial_{\tau^k} \\
 & = \sum_{k=1}^2 \sum_{\alpha = 0}^{\infty} (\rho - 1)^\alpha v_{\alpha}^{k}\circ \widetilde{\Psi}^{-1} \frac{1}{\rho} \partial_{\tau^k} \\
% & = \sum_{\alpha = 0}^{\infty} (\rho - 1)^{\alpha} \sum_{k=1}^2 v_{\alpha}^{k}\circ \widetilde{\Psi}^{-1}\frac{1}{\rho} \partial_{\tau^k} \\
& = \sum_{\alpha = 0}^{\infty} (\rho -1)^{\alpha} \tilde U_{\alpha} ,
\end{split}
\end{equation}
where
\begin{equation}\label{clear}
\tilde U_{\alpha} = \sum_{k=1}^2 v_{\alpha}^{k}\circ \widetilde{\Psi}^{-1}\frac{1}{\rho} \partial_{\tau^k}
\end{equation}
are real analytic on $U(p_0)$.
Since ${\rho^{-1}} \partial_{\tau^k}$ are parallel vector fields along every straight line radiating out from the origin and that
$v_{\alpha}^k \circ \widetilde{\Psi}^{-1}$ are constant along each straight line radiating out from the origin,
it follows from \eqref{clear} that $U_{\alpha}$ is also a parallel vector field along every straight line radiating out from the origin.\\

\textbf{Step 3.} Note that the expansion $v = \sum_{\alpha =0 }^{\infty} (\rho - 1 )^{\alpha} U_{\alpha}$ as in \eqref{PowerSeriesExapn}, with
$U_{\alpha}$ as given in \eqref{clear}, arises from a selected real analytic local parametrization $\Psi : D \rightarrow E$. Here we show the expansion is independent of the choice of the real analytic local parameterization $\Psi$.
Let $\Psi_{\sharp} : D_{\sharp} \rightarrow E$ be another real analytic local parametrization with $p_0 \in \Psi (D) \cap \Psi_{\sharp} (D_{\sharp})$.
Now, we repeat all the steps in the above argument by replacing $\Psi$ by $\Psi_{\sharp}$, and we get the associated local power series expansion $v = \sum_{\alpha =0 }^{\infty} (\rho - 1 )^{\alpha} U_{\sharp , \alpha}$. 

We show
$U_{\alpha} (p_0) = U_{\sharp , \alpha} (p_0)$ for all $\alpha \geq 0$. Take $\{e_1 , e_2\}$ be an orthonormal basis of $T_{p_0}E$. For each
$k = 1,2$, let $\widetilde{e}_k$ be the parallel vector field along the radial straight line $\{\rho p_0 : \rho > 0\}$ passing through the origin and $p_0$. For each $\alpha \geq 0$, by construction, both $U_\alpha$ and $U_{\sharp , \alpha}$ are parallel vector fields along $\{\rho p_0 : \rho > 0\}$. Then 
\begin{equation}
\begin{split}
U_{\alpha} & = A_{\alpha}^1 \widetilde{e}_1 + A_{\alpha}^2 \widetilde{e}_2, \\
U_{\sharp , \alpha} & = A_{\sharp , \alpha}^1 \widetilde{e}_1 + A_{\sharp , \alpha}^2 \widetilde{e}_2 ,
\end{split}
\end{equation}
with $A_{\alpha}^1 ,  A_{\alpha}^2 , A_{\sharp , \alpha}^1 , A_{\sharp , \alpha}^2 \in \mathbb{R}$. Thus, the identity
\begin{equation*}
\sum_{\alpha =0 }^{\infty} (\rho - 1 )^{\alpha} U_{\alpha} = \sum_{\alpha =0 }^{\infty} (\rho - 1 )^{\alpha} U_{\sharp,\alpha}
\end{equation*}
can be written as

\begin{equation*}
\Big ( \sum_{\alpha =0 }^{\infty} (\rho - 1 )^{\alpha}  A_{\alpha}^1 \Big ) \widetilde{e}_1 +\Big ( \sum_{\alpha =0 }^{\infty} (\rho - 1 )^{\alpha}  A_{\alpha}^2 \Big ) \widetilde{e}_2
= \Big ( \sum_{\alpha =0 }^{\infty} (\rho - 1 )^{\alpha}  A_{\sharp , \alpha}^1 \Big ) \widetilde{e}_1 +\Big ( \sum_{\alpha =0 }^{\infty} (\rho - 1 )^{\alpha}  A_{\sharp , \alpha}^2 \Big ) \widetilde{e}_2 .
\end{equation*}
By comparing the $\widetilde{e}_k$-components,  we get
for each $k=1,2$ that
\begin{equation*}
\sum_{\alpha =0 }^{\infty} (\rho - 1 )^{\alpha}  A_{\alpha}^k = \sum_{\alpha =0 }^{\infty} (\rho - 1 )^{\alpha}  A_{\sharp , \alpha}^k ,
\end{equation*}
which gives $A_{\alpha}^k =  A_{\sharp , \alpha}^k$, for all $k =1,2$ and all $\alpha \geq 0$. Hence, we must have $U_{\alpha} = U_{\sharp , \alpha}$ for all $\alpha \geq 0$.\\

\textbf{Step 4.} Now, since $E$ is compact, we can find a finite list $\Psi_j : D_j \rightarrow E$ of real analytic local parametrizations of $E$, with $1 \leq j \leq l$ covering $E$ so that
\begin{equation}\label{cover}
E =  \cup_{j=1}^{l} \Psi_j (D_j)
\end{equation}
and that we have the power series expansion $v = \sum_{\alpha = 0}^{\infty} (\rho - 1)^{\alpha} U_{\alpha}$ on $\mathcal{C}(\Psi_j (D_j)) \cap \mathcal{N}(\eps_j)$, with $0 < \eps_j < \eps_0$, where $U_{\alpha}$ are parallel vector fields along all the radial straight lines emitting from the origin. We have already proved that $U_{\alpha}$ are all independent of the choice of the real analytic local parametrization $\Psi_j$. Thus,
by taking $\widetilde{\eps} = \min \{\eps_j : 1 \leq j \leq l \}$, we can get a global power series expansion $v = \sum_{\alpha = 0}^{\infty} (\rho - 1)^{\alpha} U_{\alpha}$ defined on the neighborhood $\mathcal{N}(\widetilde{\eps})$ of $E$.
\end{proof}

Next, consider a general real analytic vector field $v$ on $\mathcal{N}(\eps_0)$. Recall that the distance function
$\rho$ coincides with $\big ( \frac{x^2}{a^2} + \frac{y^2}{a^2} + z^2 \big )^{\frac{1}{2}}$ , thus $\rho$ is real analytic on $\mathbb{R}^3-\{ (0,0,0) \}$. Notice that
\begin{equation*}
\begin{split}
\dd \rho & = \frac{\partial \rho}{\partial x} \dd x + \frac{\partial \rho}{\partial y}  \dd y +  \frac{\partial \rho}{\partial z}  \dd z ,\\
v  &= v^1  \frac{\partial}{\partial x} + v^2  \frac{\partial}{\partial y} + v^3  \frac{\partial}{\partial z} ,
\end{split}
\end{equation*}
where the component functions $\frac{\partial \rho}{\partial x}$ ,$\frac{\partial \rho}{\partial y}$ , $\frac{\partial \rho}{\partial z}$,
$v^1$ , $v^2$ , $v^3$ are all real analytic on $\mathcal{N}(\eps_0)$. Thus, it follows that the function
\begin{equation}
v(\rho) = \dd \rho (v) = \frac{\partial \rho}{\partial x} v^1 + \frac{\partial \rho}{\partial y} v^2 + \frac{\partial \rho}{\partial z} v^3
\end{equation}
is surely real analytic on $\mathcal{N}(\eps_0)$. Notice that the real analytic function $v(\rho)$ coincides with the $v^{\rho}$ component in the expression $v = v^{\rho} \partial_{\rho}+ v^{\phi} \partial_{\phi} + v^{\theta} \partial_{\theta}$. For this reason, from now on, we may sometimes write the real analytic function $v(\rho)$ alternatively as $v^{\rho}$.  

Next, we define the following vector field
\begin{equation}\label{ExtraExtraVeryEasy1}
w = v - v^\rho \partial_{\rho} .
\end{equation}
It is easy to check that $\partial_{\rho} = \rho^{-1} \big ( x \partial_x + y \partial_y + z \partial_z \big )$. Thus, it follows that
$\partial_{\rho}$ is real analytic on $\mathbb{R}^3-\{(0,0,0)\}$. It follows that the vector field $w$ as given by \eqref{ExtraExtraVeryEasy1} is real analytic on $\mathcal{N}(\eps_0)$. 

Observe that
\begin{equation}\label{ExtraExtraVeryEasy2}
g_{\mathbb{R}^3} ( \nabla \rho  , w ) = \dd \rho (w) = \dd \rho \big (v - v(\rho) \partial_{\rho} \big ) = v(\rho) - v(\rho) = 0,
\end{equation}
so \eqref{ExtraExtraVeryEasy2} implies that $w$ is everywhere tangential to all rescaled ellipsoids
$\{ (\rho x , \rho y , \rho z ) : (x,y,z) \in E \}$, for any $\rho \in (1-\eps_0 , 1 + \eps_0 )$. Thus, we can apply the result of
Lemma \ref{basic} to deduce that there exists some $\widetilde{\eps}_1 \in (0, \eps_0)$, together with a uniquely determined sequence $\{\widetilde{U}_{\alpha}\}_{\alpha =0}^{\infty}$ of real analytic vector fields on $\mathcal{N}(\widetilde{\eps}_1)$, such that
each $\widetilde{U}_{\alpha}$ is parallel along every straight line radiating out from the origin and is tangential to every rescaled ellipsoids $\{(\rho x , \rho y , \rho z) : (x,y,z) \in E\}$ for $1-\widetilde{\eps}_1 < \rho <1+\widetilde{\eps}_1$, and that the relation
\begin{equation}\label{previousExpansion}
w = \sum_{\alpha = 0}^{\infty} (\rho - 1)^{\alpha} \widetilde{U}_{\alpha} ,
\end{equation}
holds on  $\mathcal{N} (\widetilde{\eps}_1)$.\\

Next, take any $p_0 \in E$. Let $\Psi=(\Psi^1 , \Psi^2 , \Psi^3 ): D \rightarrow E$ be a real analytic local parametrization of $E$ for which $p_0 \in \Psi (D)$. As in the proof of Lemma \ref{basic}, consider the associated real analytic local parameterization
$\widetilde{\Psi} : (1- \eps_0 , 1 + \eps_0) \times D \rightarrow \mathcal{C} (\Psi(D))\cap \mathcal{N}(\eps_0)$  as given by
\begin{equation*}
\widetilde{\Psi} (\rho , \tau^1 , \tau^2 ) = ( \rho \Psi^1 ( \tau^1 , \tau^2 ) , \rho \Psi^2 ( \tau^1 , \tau^2 ) ,  \rho \Psi^3 ( \tau^1 , \tau^2 )  ).
\end{equation*}
Since $v^{\rho} = v(\rho )$ and $\widetilde{\Psi}$ are real analytic, it follows that the function $v^\rho \circ \widetilde{\Psi} : (1- \eps_0 , 1 + \eps_0) \times D \rightarrow \mathbb{R}^1$ is real analytic on $(1- \eps_0 , 1 + \eps_0) \times D $. We write $\widetilde{\Psi} (p_0) = ( 1 , \tau_0^1 , \tau_0^2 )$. Then, there exists some sufficiently small $\eps \in (0, \eps_0)$ such that
$(1-\eps , 1 + \eps ) \times ( \tau_0^1 - \eps , \tau_0^1 + \eps ) \times ( \tau_0^2 - \eps , \tau_0^2 + \eps ) \subset (1-\eps_0 , 1+ \eps_0 ) \times D$ and that the following relation
\begin{equation}\label{ExtraExtraExtraA5}
\begin{split}
v^{\rho} \circ \widetilde{\Psi} & = \sum_{\alpha=0}^{\infty} \sum_{a=0}^{\infty} \sum_{b=0}^{\infty} \lambda_{\alpha , a ,b }
(\tau^1 - \tau_0^1)^a (\tau^2 -\tau_0^2  )^b (\rho -1)^{\alpha} \\
& = \sum_{\alpha =0}^{\infty} (\rho-1)^{\alpha} \lambda_{\alpha} (\tau^1 , \tau^2 )
\end{split}
\end{equation}
holds on $(1-\eps , 1 + \eps ) \times ( \tau_0^1 - \eps , \tau_0^1 + \eps ) \times ( \tau_0^2 - \eps , \tau_0^2 + \eps )$. In \eqref{ExtraExtraExtraA5}, $\lambda_{\alpha , a , b} \in \mathbb{R}$, for $\alpha , a, b \in \mathbb{N}$, and
\begin{equation}
\lambda_{\alpha} (\tau^1 , \tau^2 ) = \sum_{a=0}^{\infty} \sum_{b=0}^{\infty} \lambda_{\alpha , a ,b }
(\tau^1 - \tau_0^1)^a (\tau^2 -\tau_0^2  )^b .
\end{equation}
As before, we consider the open set
\begin{equation}
U(p_0) = \widetilde{\Psi} ( (1-\eps , 1 + \eps ) \times (\tau_0^1 -\eps , \tau_0^1 + \eps) \times (\tau_0^2 -\eps , \tau_0^2 + \eps)     ) .
\end{equation}
 Based on \eqref{ExtraExtraExtraA5}, we can express $v^{\rho}$ on $U(p_0)$ as follows
 \begin{equation}\label{ExtraExtraExtraA8}
 \begin{split}
 v^{\rho} & = \sum_{\alpha = 0}^{\infty} (\rho - 1)^{\alpha} \lambda_{\alpha} \circ \widetilde{\Psi}^{-1} \\
 & = \sum_{\alpha = 0}^{\infty} (\rho - 1)^{\alpha} v^{\rho}_{\alpha} ,
 \end{split}
 \end{equation}
 where $v^{\rho}_{\alpha} = \lambda_{\alpha} \circ \widetilde{\Psi}^{-1}$ is real analytic on $U(p_0)$. Note that by construction, each
 $v^{\rho}_{\alpha}$ is constant along each straight line radiating out from the origin, and thus, we can also regard each $v^{\rho}_{\alpha}$ as a real analytic function on the open set $\Psi ( (\tau_0^1 -\eps , \tau_0^1 + \eps) \times (\tau_0^2 -\eps , \tau_0^2 + \eps)     )  $ of $E$.\\

Next, we want to see that the power series expansion $v^{\rho} = \sum_{\alpha = 0}^{\infty} (\rho - 1)^{\alpha} v^{\rho}_{\alpha} $ as obtained in \eqref{ExtraExtraExtraA8} is indeed independent of the choice of the real analytic local parametrization $\psi : D \rightarrow E$ which we choose to cover the point $p_0$. To see this, take any point $q_0 \in \Psi ( (\tau_0^1 -\eps , \tau_0^1 + \eps) \times (\tau_0^2 -\eps , \tau_0^2 + \eps)     )$, we look at the real analytic function $v^{\rho} (t q_0)$ in the real variable $t\in (1-\eps_0 , 1 + \eps_0 )$. But \eqref{ExtraExtraExtraA8} leads to the identity
\begin{equation}\label{ExtraExtraExtraA10}
v^{\rho} (t q_0) = \sum_{\alpha = 0}^{\infty} (t-1 )^{\alpha} v_{\alpha}^{\rho}(q_0) ,
\end{equation}
which holds for all $t \in (1-\eps , 1 + \eps )$. Notice that the real analytic function $v^{\rho} (t q_0)$ of a single variable $t$, which appears on the left hand side of \eqref{ExtraExtraExtraA10} is independent of the choice of the real analytic local parametrization $\Psi : D \rightarrow E$, and thus, the expression $\sum_{\alpha = 0}^{\infty} (t-1 )^{\alpha} v_{\alpha}^{\rho}(q_0)$  appearing on the right hand side of \eqref{ExtraExtraExtraA10} is really nothing but the power series expansion of the real analytic function $v^{\rho} (t q_0)$ in the variable $t$. As such, each $ v_{\alpha}^{\rho}(q_0)$ is just the coefficient in front of the $\alpha$-th power of $(t-1)$ as involved in the power series expansion of $v^{\rho} (t q_0)$ about $t =1$. Thus, this really says that
$v_{\alpha}^{\rho}(q_0)$ is really independent of the choice of the local parametrization $\Psi : D \rightarrow E$. \\

Now, we have proved the invariance of  $v^{\rho} = \sum_{\alpha = 0}^{\infty} (\rho - 1)^{\alpha} v^{\rho}_{\alpha} $ under any choice of any possible real analytic local parametrization. So, we are ready to repeat the very same argument in \textbf{Step 4} of Lemma \ref{basic} to get some $\widetilde{\eps}_2 \in (0, \eps_0)$ for which we can get a global power series expansion
\begin{equation}\label{ExtraExtraExtraA11}
v^{\rho} = \sum_{\alpha = 0}^{\infty} (\rho - 1)^{\alpha} v_{\alpha}^\rho ,
\end{equation}
on $\mathcal{N} (\widetilde{\eps}_2)$, where each real analytic function $v_{\alpha}^{\rho}$ is really constant along every single straight line radiating out from the origin. Thus, we can also regard each $v_{\alpha}^{\rho}$ as a real analytic function defined on $E$.
Finally, by taking $\widetilde{\eps} = \min \{\widetilde{\eps}_1 , \widetilde{\eps}_2  \}$, it follows from
\eqref{previousExpansion} and \eqref{ExtraExtraExtraA11} that the original real analytic vector field $v$ has the following power series expansion
\begin{equation}
v = \sum_{\alpha = 0}^{\infty} (\rho - 1)^{\alpha} v_{\alpha}^\rho \partial_{\rho} + \sum_{\alpha = 0}^{\infty} (\rho - 1)^{\alpha} \widetilde{U}_{\alpha} ,
\end{equation}
which holds on $\mathcal{N}(\widetilde{\eps})$. Thus, we have established the following result.

\begin{lemma}
Let $v$ be a real analytic vector field on $\mathcal{N}(\eps_0)$, where $0 < \eps_0 <1$. Then, it follows that there exists some sufficiently small $\widetilde{\eps} \in (0 , \eps_0)$ such that $v$ can be represented as a power series expansion on $\mathcal{N}(\widetilde{\eps})$ as follows
\begin{equation}
v = \sum_{\alpha = 0}^{\infty} (\rho - 1)^{\alpha} v_{\alpha}^\rho \partial_{\rho} + \sum_{\alpha = 0}^{\infty} (\rho - 1)^{\alpha} \widetilde{U}_{\alpha} ,
\end{equation}
where the real analytic functions $v_{\alpha}^\rho$ and real analytic vector fields $\widetilde{U}_{\alpha}$ must satisfy the following properties on $\mathcal{N} (\widetilde{\eps})$.
\begin{itemize}
\item Each $v_{\alpha}^\rho$ is constant along every single straight line radiating out from the origin of $\mathbb{R}^3$.
\item  Each $\widetilde{U}_{\alpha}$ is parallel along every straight line radiating out from the origin and is tangential to every rescaled ellipsoids $\{(\rho x , \rho y , \rho z) : (x,y,z) \in E\}$ for $1-\widetilde{\eps} < \rho <1+\widetilde{\eps}$.
\end{itemize}
\end{lemma}

In a similar fashion we can establish the existence of the power series expansion in a tubular neighborhood of the ellipsoid $E$.

We first set up some notation. For each $p \in E$, let $N_{p}$ be the outward pointing normal unit vector to the ellipsoid $E$ at $p$. For each $\sigma \in (-\eps_0 , \eps_0 )$, we look at the ``distorted" ellipsoid $\tilde E_{\sigma}$ as given by
\begin{equation}\label{Esigma}
\tilde E_{\sigma} = \big \{ p + \sigma N_p : p \in E     \big \}.
\end{equation}
We can define the following tubular neighborhood $\mathcal{T}(\eps_0)$ by taking unions of $\tilde E_{\sigma}$.
\begin{equation}\label{tubular}
\mathcal{T}(\eps_0) =  \bigcup_{\sigma \in (-\eps_0 , \eps_0 )} \tilde E_{\sigma} .
\end{equation}

\begin{lemma}\label{Asummarynotbad}
Take $E = \{ (x,y,z) \in \mathbb{R}^3 : \frac{x^2}{a^2} + \frac{y^2}{a^2} + z^2 = 1   \}$. For some
sufficiently small $\eps_0 >0$, with respect to which we can consider the tubular open neighborhood $\mathcal{T}(\eps_0)$ of $E$ as given by \eqref{tubular}. Let $v$ be a real analytic vector field on $\mathcal{T}(\eps_0)$ which is tangential to every single ``distorted" ellipsoid  $\tilde E_{\sigma} = \{ p + \sigma N_p : p \in E \}$, for all $\sigma \in (-\eps_0 , \eps_0 )$. Then, it follows that there exists some $\widetilde{\eps} \in (0 ,\eps_0 )$, together with a uniquely determined sequence
$\{U_\alpha\}_{\alpha = 0}^{\infty}$ of real analytic vector fields on $\mathcal{T}(\widetilde{\eps})$
which satisfies the following properties:
\begin{itemize}
\item For each integer $\alpha \geq 0$, $U_{\alpha}$ is parallel along every single straight line segment $\{p + \sigma N_p : \sigma \in (-\widetilde{\eps} , \widetilde{\eps} ) \}$ , for all $p \in E$.
\item For each integer $\alpha \geq 0$, $U_{\alpha}$ is tangential to every single ``distorted" ellipsoid $\tilde E_{\sigma}$, for all $\sigma \in (-\widetilde{\eps} , \widetilde{\eps} )$.
\item $v = \sum_{\alpha =0}^{\infty} \sigma^{\alpha} U_{\alpha}$ holds on $\mathcal{T}(\widetilde{\eps})$ .
\end{itemize}
\end{lemma}

\section{Forms and vectors}
Here, for completeness we establish some elementary relationships between cross products of vector fields and the wedge of the corresponding forms as well as curl and $\dd$. We first recall the definitions of the musical isomorphisms (lowering and raising the indices) as well as the definition of the Hodge star operator. We give the definitions for any Riemannian manifold $(M,g)$.  In the lemma below, we apply them to the Euclidean space $\R^3$.

If $X$ is a vector field, we can define a $1$-form, $X^\flat$, by
\[
X^\flat(Y)=g(X,Y).
\]
Then the inverse operator, $\sharp$, is defined by
\[
\alpha(Y)=g(\alpha^\sharp, Y). 
\]
Also, the Hodge $\star$ operator that takes $k$-forms to $(n-k)$-forms, is defined by

\[
\alpha\wedge \star \beta=g(\alpha, \beta)\Vol,
\]
where $\Vol$ is the volume form on the manifold in question.  
Now, we show the following lemma.
\begin{lemma}
Let $\alpha$ and $\beta$ be two $1$-forms on $\R^3$, and let $u$ be a vector field on $\R^3$.  Then the following holds
\begin{align}
\alpha ^\sharp \times \beta^\sharp&=\big(\star(\alpha \wedge \beta)\big)^\sharp,\label{crossandwedge}\\
\curl u&=(\star \dd u^\flat)^\sharp.\label{curlandd}
\end{align}
    
\end{lemma}
\begin{proof}
We can show both formulas using the Cartesian coordinates on $\R^3$.  

For \eqref{crossandwedge},
if $\alpha, \beta$ are $1$-forms, then
\[
\alpha=\alpha_i \dd x^i,\quad \beta=\beta_i \dd x^i,
\]
and using the definition of $\sharp$, we can show
\[
\alpha^\sharp=\alpha^i \partial_{x^i}=\langle \alpha^1, \alpha^2, \alpha^3\rangle, \quad
\beta^\sharp=\beta^i \partial_{x^i}=\langle \beta^1, \beta^2, \beta^3\rangle.
\]
We note that since we are using the standard metric on $\R^3$, $\alpha_i=\alpha^i$, and $\beta_i=\beta^i$ for all $i=1, 2, 3.$

Then, the left hand side in \eqref{crossandwedge} is
\[
\alpha ^\sharp \times \beta^\sharp=<\alpha^2\beta^3-\alpha^3\beta^2, \alpha^3\beta^1-\alpha^1\beta^3,\alpha^1\beta^2-\alpha^2\beta^1>.
\]
For the right hand, we first have
\[
\alpha \wedge \beta=(\alpha_i \dd x^i)\wedge (\beta_j \dd x^j)=\alpha_i \beta_j \dd x^i\wedge \dd x^j.
\]
So
\[
\alpha \wedge \beta=(\alpha_1 \beta_2-\alpha_2 \beta_1 )\dd x^1\wedge \dd x^2+(\alpha_1 \beta_3-\alpha_3 \beta_1 )\dd x^1\wedge \dd x^3+
(\alpha_2 \beta_3-\alpha_3 \beta_2 )\dd x^2\wedge \dd x^3.
\]
Using the definition of the Hodge star operator we can compute
\begin{align*}
\star( \dd x^1\wedge \dd x^2)&= dx^3\\
\star (\dd x^1\wedge \dd x^3)&=-dx^2\\
\star (\dd x^2\wedge \dd x^3)&=dx^1
\end{align*}
It follows
\[
\star (\alpha \wedge \beta)=(\alpha_1 \beta_2-\alpha_2 \beta_1 )\dd x^3-(\alpha_1 \beta_3-\alpha_3 \beta_1 )\dd x^2+
(\alpha_2 \beta_3-\alpha_3 \beta_2 )\dd x^1,
\]
which, after we apply the $\sharp$ operator, coincides with the above, as needed.

To show $\curl u=(\star \dd u^\flat)^\sharp$, we have
\[
\dd u^\flat=\sum_{i<j}(\partial_i u_j-\partial_j u_i)\dd x^i\wedge \dd x^j
\]
Then, using the definition of Hodge star we get
\[
\star \dd u^\flat=\sum_{i<j}(\partial_i u_j-\partial_j u_i)\star(\dd x^i\wedge \dd x^j)= (\partial_1 u_2-\partial_2 u_1)dx^3-(\partial_1 u_3-\partial_3 u_3)dx^2+(\partial_2 u_3-\partial_3 u_2)dx^1,
\]
which coincides with $\curl u$ after applying the $\sharp$ operator.
 \end{proof}
 \section*{Acknowledgements}
 The first author was funded in part by a grant from the Ministry of Science and Technology of Taiwan (109-2115-M-009 -009 -MY2). The second author was partially supported by a grant
from the Simons Foundation \# 585745.  The third author was partially supported by the National Science Foundation's MPS-Ascend Postdoctoral Fellowship No. 2316699.
 \bibliography{ref}
\bibliographystyle{plain}

\end{document}